\documentclass[12pt]{article}
\usepackage{epsfig}
\usepackage{amsfonts}
\usepackage{amssymb}
\usepackage{amsmath}
\usepackage{amsthm}
\usepackage{latexsym}
\usepackage{graphicx}
\usepackage{bm}
\usepackage{tabularx}
\usepackage{booktabs} 
\usepackage{enumerate}
\usepackage{caption}
\usepackage[titletoc]{appendix}
\usepackage[numbers]{natbib}
\usepackage{bbm}
\usepackage{url}
\usepackage{subfig}
\usepackage{hyperref}
\usepackage{color}
\usepackage{xcolor}
\usepackage[section]{placeins}
\usepackage[thinlines]{easytable}
\usepackage{afterpage}

\catcode`\@=11

\newcommand{\shortdot}[1]{\raisebox{-0.4pt}{$\stackrel{\bullet}{#1}$}}

\theoremstyle{plain}
\newtheorem{theorem}{Theorem}[section]

\theoremstyle{definition}

\theoremstyle{remark}

\graphicspath{ {C:/Users/doldop/Dropbox/Philip_Doldo/Mean Field/DDE System Plots/} }

\begin{document}

\title{Mean Field Queues with Delayed Information}
\author{   
Philip Doldo \\ Center for Applied Mathematics \\ Cornell University
\\ 657 Rhodes Hall, Ithaca, NY 14853 \\  pmd93@cornell.edu  \\ 
\and
  Jamol Pender \\ School of Operations Research and Information Engineering \\ Center for Applied Mathematics \\ Cornell University
\\ 228 Rhodes Hall, Ithaca, NY 14853 \\  jjp274@cornell.edu  \\ 
 }


\maketitle
\begin{abstract}
In this paper, we consider a new queueing model where queues balance themselves according to a mean field interaction with a time delay. Unlike other work with delayed information our model considers multi-server queues with customer abandonment. In this setting, our queueing model corresponds to a system of mean field interacting delay differential equations with a point of non-differentiability introduced by the finite-server and abandonment terms. We show that this system of delay differential equations exhibits a change in stability when the delay in information crosses a critical threshold.  In particular, the system exhibits periodic oscillations when the delay in information exceeds this critical threshold and we show that the threshold surprisingly does not depend on the number of queues.  This is in stark contrast to other choice based queueing models with delayed information.  We compute this critical threshold in each of the relevant parameter regions induced by the point of non-differentiability and show numerically how the critical threshold transitions through the point of non-differentiability.
\end{abstract}


\section{Introduction} \label{sec_intro}

In many queueing systems, customers are provided with information about waiting times or queue lengths which can impact their decision to wait for service. However, this information is usually not provided to customers in real time and consequently most of this information is inherently delayed. Several factors could lead to information being delayed, such as the fact that it often takes time to process and send information to customers so that customers are receiving information from some amount of time in the past by the time it actually arrives to them. Much of the delayed information literature considers models where the customers are given the length of the queue from some amount of time in the past. This information is often modeled with a constant delay, as in \citet{mitzenmacher2000useful, raina2005buffer, kamath2015car, raina2015stability, lipshutz2015existence, lipshutz2017exit, Lipshutz2018, nirenberg2018impact, novitzky2020limiting, doldo2021breaking, novitzky2019nonlinear, pender2017queues, pender2018analysis, whitt2021many}. Instead of a constant delay, a queueing model with updating information is considered in \citet{doldo2021queues, novitzky2020update}. Queueing models with random delays are considered in \citet{doldo2020multi, novitzky2020queues}, although these are typically examined through the lens of distributed delay equations which differ from standard models with the constant delay being replaced with a random variable, as discussed in \citet{doldo2021note}. Additionally, other types of information such as delayed velocity information or moving average information are considered in \citet{dong2018impact, novitzky2020limiting, pender2017queues, novitzky2019nonlinear}.

In this paper, we present a new queueing model that uses the delayed queue length information in a new way.  In particular, we consider a model where customers are informed by mean-field information. To illustrate what we mean by this, if customers are informed by standard queue length information from $\Delta> 0$ time units in the past and $Q_i(t)$ is the length of the $i^{\text{th}}$ queue in the system of $N$ queues at time $t$, then they are given the information $$Q_i(t - \Delta)$$ for $i = 1, ..., N$. On the other hand, if customers are provided with mean field queue length information, then they are provided with $$(Q_i(t- \Delta))^p - \frac{1}{N} \sum_{j=1}^{N} (Q_i(t - \Delta))^p$$ for $i=1, ..., N$ where $p \geq 0$ is a parameter. Note that when $p=1$, this reduces to the difference of the delayed queue length and the mean of all of the delayed queue lengths. When $p=0$, the information associated with each queue is identical and thus it is equivalent to giving the customer no information at all which will result in the customer choosing any of the queues with equal probability. To the authors' knowledge, there exists only two papers that explore the possibility of using a "mean field" type interaction to perform load balancing of queueing systems \cite{dawson2005balancing, bobbio2008analysis}.  However, this work does not consider the fact that the information about the queue length might be delayed in realistic situations.  Thus, our goal is to explore how delayed queue length information impacts the dynamics of the queueing system with a mean field load balancing interaction. 

Although mean field queues have only been considered in \cite{dawson2005balancing, bobbio2008analysis} for load balancing considerations, mean field dynamics of stochastic systems and deterministic systems have been studied in a variety of application settings.  For example in the financial mathematics literature there is a large effort to study mean field games and mean field stochastic differential equations \cite{buckdahn2009mean, buckdahn2017mean, carmona2015probabilistic, carmona2016mean, carmona2017mean, heesen2021fluctuation, lacker2019mean, nutz2018mean, lauriere2020convergence, possamai2021non, luo2021laplace}.  However, this mean field interaction approach has not been applied heavily in the context of queueing theory.  Thus, one of our goals in this work is to make the appeal to the queueing and applied mathematics communities that the mean field interaction approach has some value for load balancing applications in queueing systems.

 \subsection{Main Contributions of Paper}

The contributions of this work can be summarized as follows:    
\begin{itemize}
\item We develop a new stochastic queueing model that incorporates mean field behavior with delayed information and uses a finite number of servers.   
\item We solve for the equilibrium solution and compute the critical delay at which the system exhibits a change in stability.  We find that the critical threshold for stability does not depend on the number of queues, which highlights a stark difference between queueing models with mean field interactions and one with choice models.   
\item Finally, we provide a numerical examination of how the critical threshold varies as a point of non-differentiability in our model is approached and crossed.  This provides insight on how the thresholds in the two different parameter regions are connected.  
\end{itemize} 


\subsection{Organization of Paper}

The remainder of this paper is organized as follows. In Section \ref{model_section}, we introduce the stochastic queueing model and the corresponding fluid limit which gives rise to the system of delay differential equations that we analyze the dynamics of. In Section \ref{sec_Hopf} we perform the stability analysis of the queueing system by computing the equilibrium and critical delay of the system. We note that we do this for two parameter cases which are relevant due to the finite-server and abandonment terms in our model.  In Section \ref{conclusion_section} we give concluding remarks and discuss possible extensions.

\section{Mean Field Queues } \label{model_section}

In this section, we present a new stochastic queueing model where we have $N$ multi-server queues operating in parallel and customers choose which station to join via a mean field interaction model that depends on the queue lengths.  Additionally, customers in the queue that are not immediately served have the possibility of abandoning in which case they exit the queue without receiving service.   We assume that the total possible arrival rate to each queue is $\lambda > 0$, the service rate at each queue is given by $\mu > 0$, and the abandonment rate is given by $\beta > 0$.   

To model the customer arrival process, we incorporates a use a mean field interaction approach.  More specifically, we use a delayed mean field queue length information.  Thus, the probability that a customer joins the $i^{th}$ queue is given by the following expression

 \begin{eqnarray}
p_i \left( Q_i(t), \overline{Q}^{(N)}(t), \Delta \right) &=& p_i \left( Q_i^p(t-\Delta) - \overline{Q}^{(N,p)}(t-\Delta) \right) \\
 &=& p_i \left( Q_i^p(t-\Delta) - \frac{1}{N} \sum^{N}_{j=1} Q_j^p(t-\Delta) \right) 
\end{eqnarray}
where $\overline{Q}^{(N,p)}(t) =  \frac{1}{N} \sum^{N}_{j=1} Q^p_j(t) $.   It is important to note two things.  The first is that $\overline{Q}^{(N,p)}(t) =  \frac{1}{N} \sum^{N}_{j=1} Q^p_j(t) $ is the empirical $p^{th}$ moment of the queue length processes.  Second, the mean field interaction is depends on the parameter $p$.  When $p=1$, we have a traditional mean field interaction term.  When $p$ is increased, the interaction of the queues is increased and when $p$ is decreased the interaction is decreased.  In fact when $p=0$, there is no interaction between the queues at all and all queues are independent Erlang-A queues.  It is common that one might use the logistic function as a function to weight the rate that a customer will choose the $i^{th}$ queue.  This yields the following expression for the rate at which customers will join the $i^{th}$ queue
 \begin{eqnarray}
p_i \left( Q_i(t), \Delta, p \right) &=& \frac{1}{\gamma + e^{\theta \left( Q^p_i(t-\Delta) - \frac{1}{N} \sum^{N}_{j=1} Q^p_j(t-\Delta) \right)}}.
\end{eqnarray} However, in what follows, we will use some differentiable function $f: \mathbb{R} \to \mathbb{R}$, so that the logistic function is just a special case of what we consider where $$f(x) = \frac{1}{ \gamma + \exp(\theta x)}.$$

Another common function that could be used is the probit function, which is the complementary cumulative distribution function of the standard Gaussian distribution i.e.
$$f(x) = 1 - \Phi(x) = 1 - \int^{x}_{-\infty} \frac{1}{\sqrt{2\pi}} e^{-y^2/2}dy = \int^{\infty}_{x} \frac{1}{\sqrt{2\pi}} e^{-y^2/2}dy =.$$

A stochastic model for our $N$-dimensional queueing network for $t \geq 0$ is given by 

\begin{eqnarray}
Q_i(t) &=& Q _i([-\Delta,0])  +  \Pi^a_{i} \left( \int^{t}_{0} \lambda f \left( Q^p_i(s-\Delta) - \frac{1}{N} \sum^{N}_{j=1} Q^p_j(s-\Delta) \right)  ds \right) \nonumber \\
&-&  \Pi^d_{i} \left( \int^{t}_{0} \mu ( Q_i(s) \wedge c) ds \right) - \Pi_i^r \left( \int_0^t \beta (Q_i(s) - c)^+  \right) \label{cdsqnoeta}
\end{eqnarray}
 where each $\Pi(\cdot)$ is a unit rate Poisson process and  $Q_i(s) = \varphi_i(s)$ for all $s \in [- \Delta,0]$.  In this model, for the $i^{th}$ queue, we have that 
  \begin{equation}
 \Pi^a_{i} \left( \int^{t}_{0}  \lambda f \left( Q^p_i(s-\Delta) - \frac{1}{N} \sum^{N}_{j=1} Q^p_j(s-\Delta) \right) ds \right)
 \end{equation}
  counts the number of customers that decide to join the $i^{th}$ queue in the time interval $(0,t]$.  Note that the rate depends on the queue length at time $t - \Delta$ and not time $t$, hence representing a constant lag in information of size $\Delta$.  Similarly 
  \begin{equation}
  \Pi^d_{i} \left( \int^{t}_{0} \mu (Q_i(s) \wedge c) ds \right)
  \end{equation}
   counts the number of customers that depart the $i^{th}$ queue having received service from any of the $c$ servers in the time interval $(0,t]$. However, in contrast to the arrival process, the service process depends on the current real-time queue length and does not depend on the past queue length times in any way.  Additionally, 
     \begin{equation}
  \Pi^r_{i} \left( \int^{t}_{0} \beta (Q_i(s) - c)^+ ds \right)
  \end{equation}
  counts the number of customers that abandon their position in the $i^{th}$ queue without receiving service due to waiting too long to receive service in the time interval $(0, t]$. We note that the possibility of customer abandonment only occurs when the number of customers in the queue exceeds the number of servers $c$. This abandonment process also depends on the current real-time queue length and thus the arrival process is the only part of the model that depends on delayed queue lengths.  Except for the arrival process, our queueing model has a stark resemblance to the Erlang-A queueing model, which has been extensively analyzed in \cite{mandelbaum1998strong, mandelbaum2004, zeltyn2005, engblom2014approximations, daw2019new, shah2019using, massey2018dynamic}.  The Erlang-A queueing model is a canonical model for service systems, however, it is well known that the points of non-differentiability are challenging to overcome, see for example \cite{ko2013critically, massey2013gaussian, pender2017approximating, pender2017sampling}.

   \subsection{Fluid Limits}
In many service systems, the arrival rate of customers is high.  For example in cloud computing systems, many customers are accessing their data or information at any given time \cite{pender2016law}.  In amusement parks such as Disneyland there are thousands of customers moving around the park and deciding on which ride they should join \cite{nirenberg2018impact}.  Motivated by the large number of customers, we introduce the following scaled queue length process by a parameter $\eta$ 
  \begin{eqnarray}
Q^{\eta}_i(t) &=&  Q^{\eta} _i([- \Delta,0]) +  \frac{1}{\eta}\Pi^a_{i} \left( \eta \int^{t}_{0} \lambda f \left( Q^{p,\eta}_i(s-\Delta) - \frac{1}{N} \sum^{N}_{j=1} Q^{p,\eta}_j(s-\Delta) \right)  ds \right) \nonumber \\ &-&  \frac{1}{\eta}\Pi^d_{i} \left(\eta \int^{t}_{0} \mu (Q^{\eta}_i(s)  \wedge c) ds \right) -  \frac{1}{\eta}\Pi^r_{i} \left(\eta \int^{t}_{0} \beta (Q^{\eta}_i(s)  - c)^+ ds \right). \label{cdsqeta}
\end{eqnarray}

Note that we scale the rates of both Poisson processes.  This is different from the many-server of Halfin-Whitt scaling, which would only scale the arrival rate and number of servers in the case of a multi-server queue.  Letting the scaling parameter $\eta$ go to infinity gives us our first result.

\begin{theorem}\label{fluidlimit}
Let $\varphi_i(s)$ be a Lipschitz continuous function defined on the interval $[-\Delta,0]$.   If $Q^{\eta} _i(s) \to \varphi_i(s)$ almost surely for all $s \in [- \Delta,0]$ and for all  $1 \leq i \leq N$, then the sequence of stochastic processes $\{ Q^{\eta}(t) = (Q^{\eta}_1(t),Q^{\eta}_2(t), ..., Q^{\eta}_N(t)  \}_{\eta \in \mathbb{N}}$ converges almost surely and uniformly on compact sets of time to the functional differential equation $(q(t) = (q_1(t),q_2(t), ... , q_N(t))$ where
\begin{eqnarray}
\shortdot{q}_i(t) &=& \lambda f \left( q^p_i(t-\Delta) - \frac{1}{N} \sum^{N}_{j=1} q^p_j(t-\Delta) \right) - \mu \cdot ( q_i(t) \wedge c) - \beta \cdot (q_i(t) - c)^+  
\end{eqnarray}
and $q_i(s) = \varphi_i(s)$ for all $s \in [-\Delta,0]$ and for all  $1 \leq i \leq N$.
\end{theorem}

\begin{proof}
The proof is similar to the proof found in \citet{pender2020stochastic}.
\end{proof}



\section{Hopf Bifurcations in the Mean Field Model} \label{sec_Hopf}

We can more compactly write our system of delay differential equations as

\begin{eqnarray}
\overset{\bullet}{q}_i(t) &=& \lambda f(x_i) - \mu \cdot (q_i(t) \wedge c) - \beta \cdot (q_i(t) - c)^+ , \hspace{5mm} i = 1, ..., N \label{mean_field_system}
\end{eqnarray}
 where $$x_i :=  (q_i(t- \Delta))^p - \frac{1}{N} \sum_{j=1}^{N} (q_i(t - \Delta))^p.$$

\noindent We aim to analyze the stability of our system of delay differential equations \ref{mean_field_system} by finding an equilibrium solution and linearizing the system about this equilibrium solution. However, the service and abandonment terms introduce a point of non-differentiability so that such a linearization will not always be valid. To get a better sense of this how these terms affect the behavior of the system, we look at how the system simplifies when $q_i(t) \leq c$ or $q_i(t) > c$ for $i = 1, ..., N$. In the former case, the system \ref{mean_field_system} reduces to 
\begin{eqnarray}
\overset{\bullet}{q}_i(t) &=& \lambda f(x_i) - \mu \cdot q_i(t) , \hspace{5mm} i = 1, ..., N
\end{eqnarray}
which is equivalent to the infinite-server case (in which abandonment never occurs). In the latter case, system \ref{mean_field_system} becomes 
\begin{eqnarray}
\overset{\bullet}{q}_i(t) &=& \lambda f(x_i)  - \beta \cdot q_i(t) - \mu c + \beta c , \hspace{5mm} i = 1, ..., N
\end{eqnarray}
which can be viewed as an infinite-server system where the abandonment rate parameter $\beta$ from our original model now takes on the role of the service rate and the system is shifted according to the constant term $-\mu c + \beta c$ that is present.

Taking these cases into consideration, Theorem \ref{equilibrium_theorem} makes clear the parameter regions of interest and their corresponding equilibrium solutions.

\begin{theorem}
The system of equations given in Equation \ref{mean_field_system} has an equilibrium point at $q_1 = q_2 = \cdots = q_N = q^*$ where $$q^* = \frac{\lambda f(0)}{\mu} \hspace{5mm} \text{if} \hspace{5mm} \lambda f(0) \leq \mu c $$ and $$q^* = \frac{\lambda f(0) - \mu c + \beta c}{\beta} \hspace{5mm} \text{if} \hspace{5mm} \lambda f(0) > \mu c.$$
\label{equilibrium_theorem}
\end{theorem}

\begin{proof}
When $q_i(t) = q_i(t-\Delta) = q^*$ for all $t$, we have that $$x_i = (q^*)^p - \frac{1}{N} \sum_{j=1}^{N}(q^*)^p = 0$$ for all $i$ and for all $t$. Thus, each equation in system \ref{mean_field_system} becomes 
\begin{eqnarray}
0 &=& \lambda f(0) - \mu \cdot (q^* \wedge c) - \beta \cdot (q^* - c)^{+} \label{equilibrium_eqn}
\end{eqnarray} if $q^*$ is indeed an equilibrium point. We verify that the right-hand side of Equation \ref{equilibrium_eqn} evaluates to $0$ for the values of $q^*$ and corresponding conditions stated in the theorem.

If $\lambda f(0) \leq \mu c$, then we have
\begin{align}
\overset{\bullet}{q}_i(t) &= \lambda f(0) - \mu \cdot \left( \frac{\lambda f(0)}{\mu} \wedge c \right) - \beta \cdot \left(  \frac{\lambda f(0)}{\mu} - c \right)^+\\
&= \lambda f(0) - \mu \left( \frac{\lambda f(0)}{\mu} \right) - \beta \cdot 0\\
&= \lambda f(0) - \lambda f(0)\\
&= 0.
\end{align}
Thus, $q_i = \frac{\lambda f(0)}{\mu}$ for $i = 1, ..., N$ is an equilibrium point when $\lambda f(0) \leq \mu c$.

If $\lambda f(0) > \mu c$, then we have 
\begin{align}
\overset{\bullet}{q}_i(t) &= \lambda f(0) - \mu \cdot \left( \frac{\lambda f(0) - \mu c + \beta c}{\beta} \wedge c \right) - \beta \cdot \left(  \frac{\lambda f(0) - \mu c + \beta c}{\beta} - c \right)^+\\
&= \lambda f(0) - \mu \cdot \left( \frac{\lambda f(0) - \mu c}{\beta} + c \wedge c \right) - \beta \cdot \left(  \frac{\lambda f(0) - \mu c }{\beta} \right)^+\\
&= \lambda f(0) - \mu c - \beta \frac{\lambda f(0) - \mu c}{\beta}\\
&= 0.
\end{align}
Thus, $q_i = \frac{\lambda f(0) - \mu c + \beta c}{\beta}$ for $i = 1, ..., N$ is an equilibrium point when $\lambda f(0) > \mu c$.

\end{proof}


We now have an equilibrium solution for each of the two parameter cases of interest, namely $\lambda f(0) \leq \mu c$ and $\lambda f(0) > \mu c$. However, moving forward we make the inequality strict when discussing the former region as the linearization analysis may break down at the point of non-differentiability $\lambda f(0) = \mu c$. Next, in Theorem \ref{delta_cr_theorem} we present the critical value of the delay parameter $\Delta$ at which our system changes stability in each of the parameter regions $\lambda f(0) < \mu c$ and $\lambda f(0) > \mu c$. The key practical takeaway of this result is that if the delay in information is larger than the critical value relevant to the parameter region of interest, then the queue lengths will oscillate about the corresponding equilibrium point rather than converging to it.

\begin{theorem}
The system of equations given in Equation \ref{mean_field_system} exhibits a change in stability at a critical value of $\Delta_{cr}$ which has the following expressions 

$$\Delta_{\text{cr}} = \frac{\mu^{p-1} \arccos \left(  \frac{\mu^p}{p \lambda^p f(0)^{p-1} f'(0)} \right)}{\sqrt{p^2 \lambda^{2 p} (f(0))^{2p -2} (f'(0))^2 - \mu^{2p}}} \hspace{5mm} \text{if} \hspace{5mm} \lambda f(0) < \mu c $$ and $$\Delta_{\text{cr}} = \frac{\beta^{p-1} \arccos \left( \frac{\beta^p}{p \lambda f'(0) \left( \lambda f(0) - \mu c + \beta c \right)^{p-1}}{} \right)}{\sqrt{p^2 \lambda^2 \left( f'(0) \right)^2 \left( \lambda f(0) - \mu c + \beta c \right)^{2p-2} - \beta^{2p}}} \hspace{5mm} \text{if} \hspace{5mm} \lambda f(0) > \mu c.$$


\label{delta_cr_theorem}
 \end{theorem}

\begin{proof}
From Theorem \ref{equilibrium_theorem}, we obtained an equilibrium solution $q_1 = \cdots = q_N = q^*$ in each of the two parameter regions of interest, namely $\lambda f(0) < \mu c$ and $\lambda f(0) > \mu c$. We proceed by linearizing the system \ref{mean_field_system} about the equilibrium in each case. Let $$q_i(t) = q^* + u_i(t) \hspace{5mm} i=1,...,N$$ where $$q^* = \frac{\lambda f(0)}{\mu} \hspace{5mm} \text{if} \hspace{5mm} \lambda f(0) < \mu c $$ and $$q^* = \frac{\lambda f(0) - \mu c + \beta c}{\beta} \hspace{5mm} \text{if} \hspace{5mm} \lambda f(0) > \mu c.$$ 
\noindent Applying this change of variables, the system becomes
\begin{align*}
\overset{\bullet}{u}_i(t) &= F(u_i(t), u_1(t - \Delta), ..., u_N(t- \Delta))\\
&:= F(u_i, u_{\Delta 1}, ..., u_{\Delta N})\\
&= \lambda f \left( \left(  q^* + u_{\Delta i} \right)^p - \frac{1}{N} \sum_{j=1}^{N} \left( q^* + u_{\Delta j} \right)^p \right) - \mu \left( (q^* + u_i)  \wedge c \right) - \beta \left( u_i + q^* - c \right)^+.
\end{align*} 
We want to consider the linearized system for $u_1$ and $u_2$ where we assume that $u_1$ and $u_2$ are initially arbitrarily small and nonzero with the understanding that our linearization is only necessarily a good approximation to the nonlinear system for sufficiently small values of $t$. In light of this consideration, we examine how the form of system simplifies in each case. First, when $\lambda f(0) < \mu c$ we have 
\begin{align*}
\overset{\bullet}{u}_i(t) &= \lambda f \left( \left(  \frac{\lambda f(0)}{\mu} + u_{\Delta i} \right)^p - \frac{1}{N} \sum_{j=1}^{N} \left( \frac{\lambda f(0)}{\mu} + u_{\Delta j} \right)^p \right) - \mu \left( \left( \frac{\lambda f(0)}{\mu} + u_i \right)  \wedge c \right)\\
&- \beta \left( \frac{\lambda f(0)}{\mu} + u_i - c \right)^+\\
&= \lambda f \left( \left(  \frac{\lambda f(0)}{\mu} + u_{\Delta i} \right)^p - \frac{1}{N} \sum_{j=1}^{N} \left( \frac{\lambda f(0)}{\mu} + u_{\Delta j} \right)^p \right) - \mu \left(  \frac{\lambda f(0)}{\mu} + u_i \right)\\
&= \lambda f \left( \left(  \frac{\lambda f(0)}{\mu} + u_{\Delta i} \right)^p - \frac{1}{N} \sum_{j=1}^{N} \left( \frac{\lambda f(0)}{\mu} + u_{\Delta j} \right)^p \right) - \mu u_i(t) - \lambda f(0)
\end{align*} where we note that making the inequality strict $\lambda f(0) < 
\mu c$ was important to ensure that $\frac{\lambda f(0)}{\mu} + u_i \leq c$ for some arbitrarily small (but nonzero) $u_i$. When $\lambda f(0) > \mu c$ we have 
\begin{align*}
\overset{\bullet}{u}_i(t) &= \lambda f \left( \left(  \frac{\lambda f(0) - \mu c + \beta c}{\beta} + u_{\Delta i} \right)^p - \frac{1}{N} \sum_{j=1}^{N} \left( \frac{\lambda f(0) - \mu c + \beta c}{\beta} + u_{\Delta j} \right)^p \right)\\
&- \mu \left( \left( \frac{\lambda f(0) - \mu c + \beta c}{\beta} + u_i \right)  \wedge c \right) - \beta \left( \frac{\lambda f(0) - \mu c + \beta c}{\beta} + u_i - c \right)^+\\
&= \lambda f \left( \left(  \frac{\lambda f(0) - \mu c + \beta c}{\beta} + u_{\Delta i} \right)^p - \frac{1}{N} \sum_{j=1}^{N} \left( \frac{\lambda f(0) - \mu c + \beta c}{\beta} + u_{\Delta j} \right)^p \right)\\
&- \mu c - \beta \left( \frac{\lambda f(0) - \mu c + \beta c}{\beta} + u_i - c \right)\\
&= \lambda f \left( \left(  \frac{\lambda f(0) - \mu c + \beta c}{\beta} + u_{\Delta i} \right)^p - \frac{1}{N} \sum_{j=1}^{N} \left( \frac{\lambda f(0) - \mu c + \beta c}{\beta} + u_{\Delta j} \right)^p \right) - \beta u_i(t) - \lambda f(0).
\end{align*}
We notice that the form of the system is the same for each parameter region except for the equilibrium solution appearing inside $f$ as well as the parameter on the linear non-delayed term.  For the $i^{\text{th}}$ equation, we have $$\frac{\partial F}{\partial u_i} =  
\begin{cases} 
       -\mu, & \lambda f(0) < \mu c \\
       -\beta, & \lambda f(0) > \mu c 
 \end{cases}$$ $$\frac{\partial F}{\partial u_{\Delta i}} \Bigg|_{(u_i, u_{\Delta 1}, ..., u_{\Delta N}) = (0, 0, ..., 0)} =
 \begin{cases} 
        \lambda f'(0) \cdot \frac{N - 1}{N} \cdot p \left( \frac{\lambda f(0)}{\mu}  \right)^{p-1}, & \lambda f(0) < \mu c \\
        \lambda f'(0) \cdot \frac{N - 1}{N} \cdot p \left( \frac{\lambda f(0) - \mu c + \beta c}{\beta}  \right)^{p-1}, & \lambda f(0) > \mu c 
 \end{cases}
$$ $$\frac{\partial F}{\partial u_{\Delta j}} \Bigg|_{(u_i, u_{\Delta 1}, ..., u_{\Delta N}) = (0, 0, ..., 0)} =
 \begin{cases} 
        \lambda f'(0) \cdot \left(- \frac{1}{N}  \right) \cdot p \left( \frac{\lambda f(0)}{\mu} \right)^{p-1}, & \lambda f(0) < \mu c \\
        \lambda f'(0) \cdot \left(- \frac{1}{N}  \right) \cdot p \left( \frac{\lambda f(0) - \mu c + \beta c}{\beta} \right)^{p-1}, & \lambda f(0) > \mu c 
 \end{cases}$$ for $j \neq i$.

After linearizing the system, we want to plug in $e^{rt}$ and solve for the critical delay. Before doing this, we want to uncouple our system. Our linearized system is currently in the form $$\overset{\bullet}{u} = A u_{\Delta} - \theta u$$ where $$\theta := \begin{cases}
\mu, & \lambda f(0) < \mu c\\
\beta, & \lambda f(0) > \mu c
\end{cases}$$ and $$A = C \begin{bmatrix}
\frac{N-1}{N} & - \frac{1}{N}  & - \frac{1}{N}  \dots  - \frac{1}{N}\\
- \frac{1}{N} & \frac{N-1}{N} & - \frac{1}{N} \dots  - \frac{1}{N}\\
- \frac{1}{N} & - \frac{1}{N} & \frac{N-1}{N} \dots - \frac{1}{N}\\
\vdots & & \ddots   \vdots\\
- \frac{1}{N} & - \frac{1}{N} & - \frac{1}{N} \dots \frac{N-1}{N}
\end{bmatrix}$$ where $$C :=  \begin{cases} 
        p \lambda f'(0) \left( \frac{\lambda f(0)}{\mu} \right)^{p - 1}, & \lambda f(0) < \mu c \\
        p \lambda f'(0) \left( \frac{\lambda f(0) - \mu c + \beta c}{\beta} \right)^{p - 1}, & \lambda f(0) > \mu c 
 \end{cases}.$$

The matrix $A$ has eigenvalues $0$ (with multiplicity 1) and $C$ (with multiplicity $N-1$). Let $E$ be the matrix whose columns are the eigenvectors of $A$ (start with the eigenvector corresponding to $0$ for ease of notation) so that $$AE = ED$$ where $$D = \begin{bmatrix}
0\\
& C\\
& & C\\
& & & \ddots\\
& & & & C
\end{bmatrix}.$$ Now we let $u = Ev$ and substitute this into our system. $$\overset{\bullet}{u} = A u_{\Delta} - \theta u$$ $$\iff$$ $$E \overset{\bullet}{v} = AE v_{\Delta} - \theta E v$$ $$\iff$$ $$E \overset{\bullet}{v} = ED v_{\Delta} - \theta E v$$ $$\iff$$ $$\overset{\bullet}{v} = D v_{\Delta} - \theta v$$ Thus, we have our uncoupled system. Less compactly written, the system is 

\begin{align*}
\overset{\bullet}{v}_1 &= - \theta v_1\\
\overset{\bullet}{v}_2 &= C v_{\Delta 2} - \theta v_2\\
\vdots\\
\overset{\bullet}{v}_N &= C v_{\Delta N} - \theta v_N.
\end{align*}

The first equation in the system has solution $v_1(t) = \tilde{c} \exp(-\theta t)$ which decays with time since $\theta > 0$ in both cases, so this first equation is not essential to the stability analysis as its solution is always stable. Now look at the $j^{\text{th}}$ equation and let $v_j = e^{rt}$ for $2 \leq j \leq N$. This gives us the characteristic equation $$r - C e^{- r \Delta} + \theta = 0.$$ Now let $r = i \omega$ to find where the stability changes.

$$i \omega - C \cos(\omega \Delta) + i C \sin(\omega \Delta) + \theta = 0$$ $$\sin(\omega \Delta) = - \frac{\omega}{C} \hspace{5mm} \text{and} \hspace{5mm} \cos(\omega \Delta) = \frac{\theta}{C}.$$

$$\omega = \sqrt{C^2 - \theta^2}$$ and $$\Delta = \frac{\arccos \left(  \frac{\theta}{C} \right) }{\omega}$$

\noindent Thus, it follows that $$\Delta_{\text{cr}} = 
\frac{\mu^{p-1} \arccos \left(  \frac{\mu^p}{p \lambda^p f(0)^{p-1} f'(0)} \right)}{\sqrt{p^2 \lambda^{2 p} (f(0))^{2p -2} (f'(0))^2 - \mu^{2p}}} \hspace{5mm} \text{if} \hspace{5mm} \lambda f(0) < \mu c$$ and 
$$\Delta_{\text{cr}} = \frac{\beta^{p-1} \arccos \left( \frac{\beta^p}{p \lambda f'(0) \left( \lambda f(0) - \mu c + \beta c \right)^{p-1}}{} \right)}{\sqrt{p^2 \lambda^2 \left( f'(0) \right)^2 \left( \lambda f(0) - \mu c + \beta c \right)^{2p-2} - \beta^{2p}}} \hspace{5mm} \text{if} \hspace{5mm} \lambda f(0) > \mu c.
$$

\end{proof}



\subsection{Numerical Experiments of Mean Field Queueing Models}

Now that we have calculated the critical delay corresponding to each of the two relevant parameter regions, we numerically verify our results in the figures below using $f(x) = \frac{1}{1 + \exp(\theta x)}$ and $f(x) = \int^{\infty}_{x} \frac{1}{\sqrt{2\pi}} e^{-y^2/2}dy$. In Figures \ref{fig1} and \ref{fig2}, we consider examples where our system contains $N=2$ queues with $f(x) = \frac{1}{1 + \exp(\theta x)}$. In Figure \ref{fig1} we consider the parameter case where $\lambda f(0) < \mu c$ when the delay parameter $\Delta$ is both below and above the critical delay value. In the former (stable) case, we see that the queue lengths approach the equilibrium solution whereas in the latter (unstable) case the queue lengths oscillate about the equilibrium solution. In the unstable case, the oscillations of the queue length approach a limiting amplitude (which becomes more apparent after integrating for longer time).The phase diagrams included plot the derivative with respect to time of the queue lengths against the queue lengths and one can easily see the convergence to the equilibrium point in the stable case the existence of a periodic orbit in the unstable case. Figure \ref{fig2} considers the analogous information in the other parameter case where $\lambda f(0) > \mu c$. Similarly, Figures \ref{fig3} and \ref{fig4} consider both the stable and unstable cases for each parameter region when the system has $N = 3$ queues. In Figures \ref{fig7} through \ref{fig10} we consider the same but for $f(x) =  \int^{\infty}_{x} \frac{1}{\sqrt{2\pi}} e^{-y^2/2}dy$.

Our formula for the critical delay is different depending on which parameter region is of interest and recall that this formula is only valid away from the point of non-differentiability that is present at the transition between these two parameter regions. In light of this, it is unclear in what way the value of the critical delay may vary as this transition region is approached and crossed. We can explore this numerically by varying $c$ while keeping all other parameters fixed and numerically approximating the critical delay for each value of $c$ considered. The results of this procedure are demonstrated in Figures \ref{fig5} and \ref{fig6} for $f(x) = \frac{1}{1 + \exp(\theta x)}$ and Figures \ref{fig11} and \ref{fig12} for $f(x) = \int^{\infty}_{x} \frac{1}{\sqrt{2\pi}} e^{-y^2/2}dy$. We see that the value of the critical delay appears to smoothly vary as the transition region is crossed in each of the cases considered.

\begin{figure}[ht!]
\hspace{-15mm} \includegraphics[scale=.33]{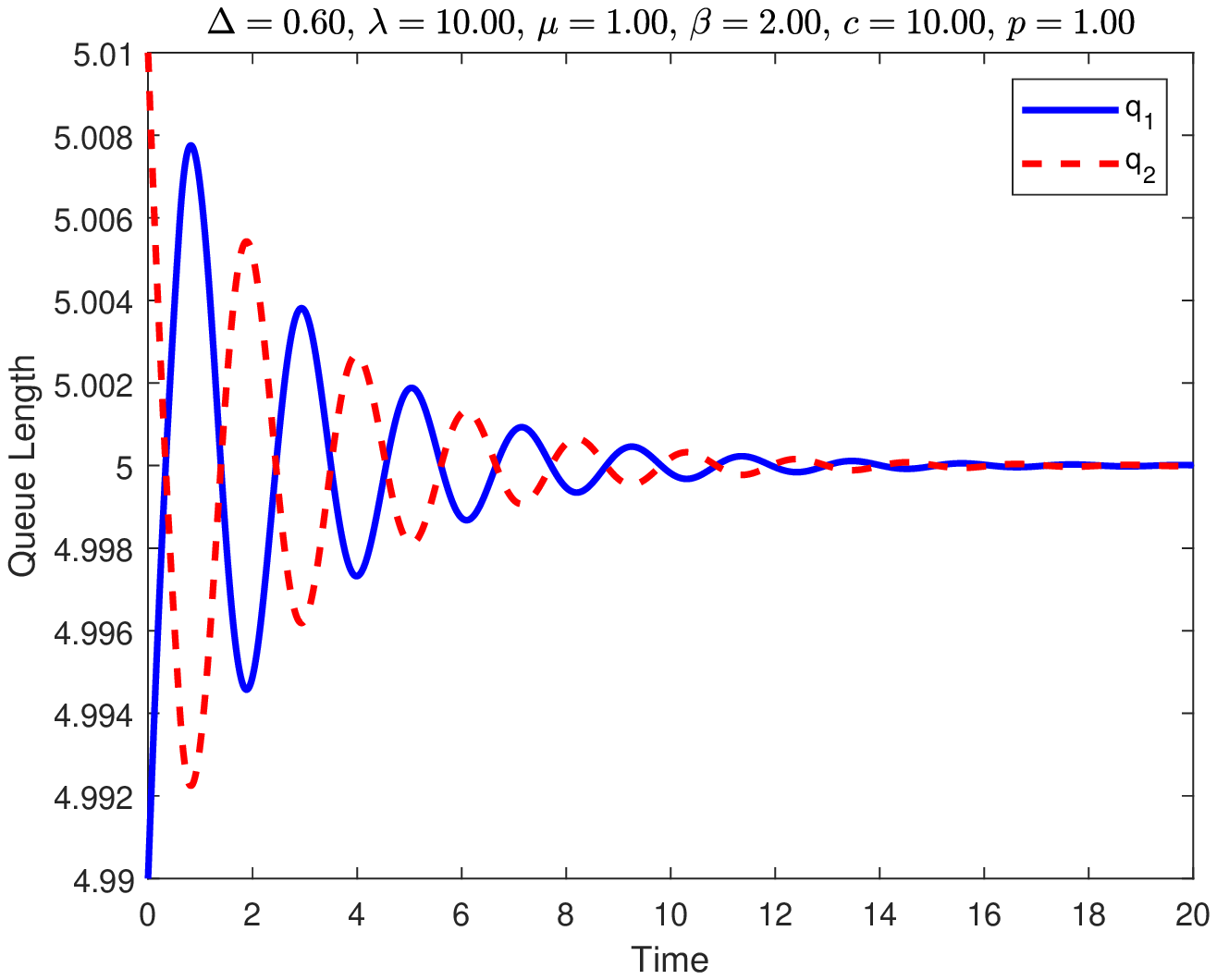}\includegraphics[scale=.33]{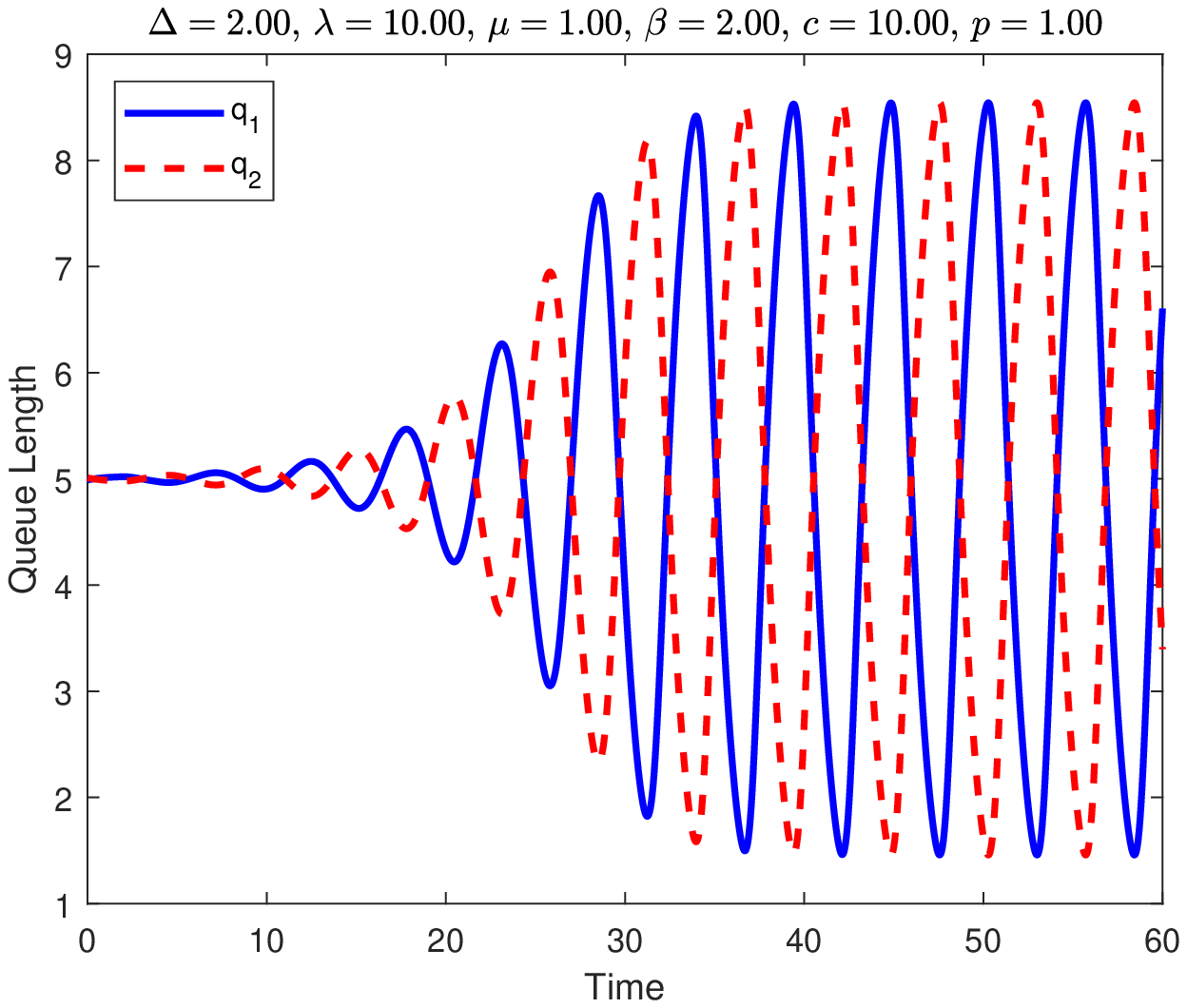}\includegraphics[scale=.33]{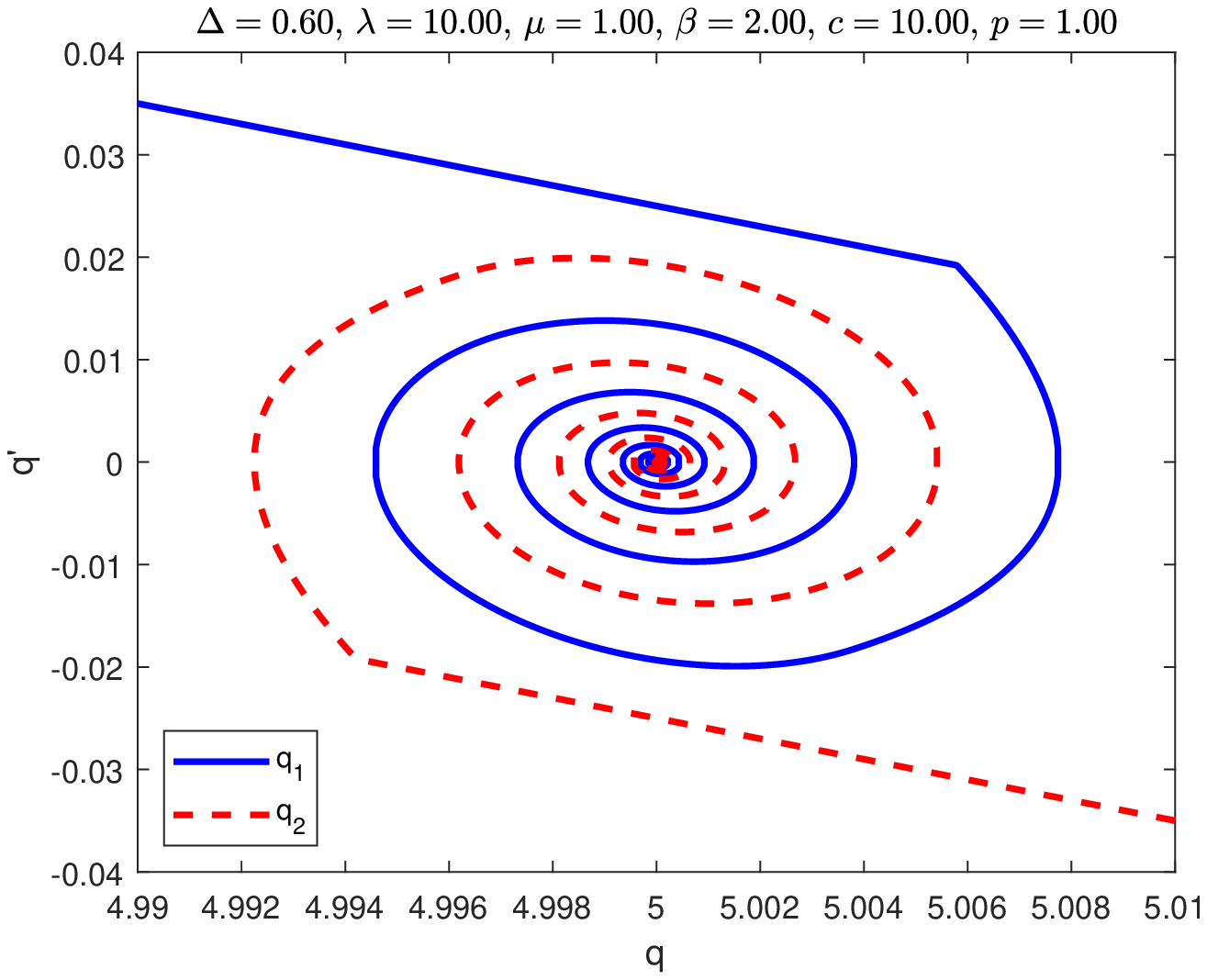}\includegraphics[scale=.33]{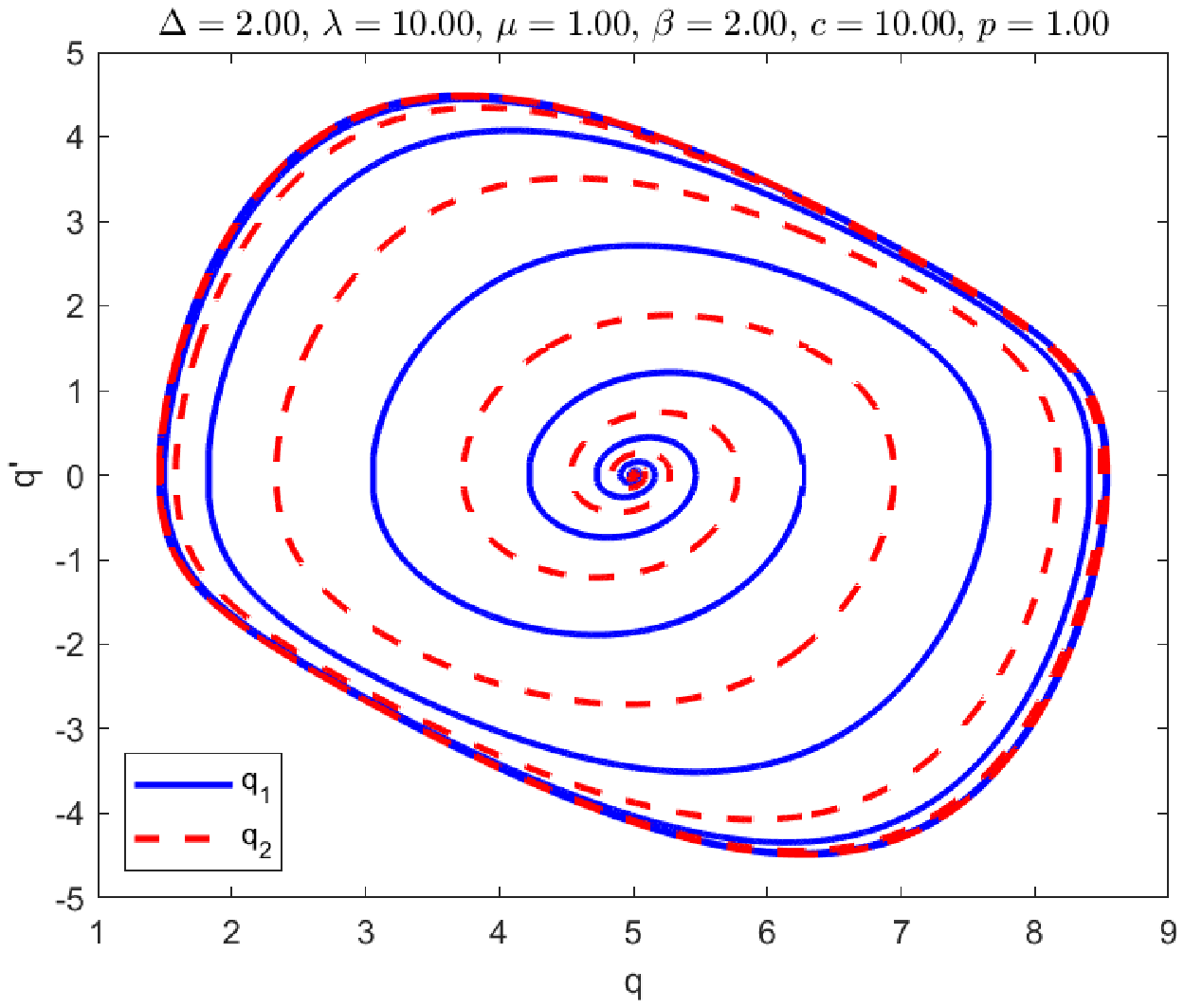}
\caption{Before and after the change in stability in the case where $\lambda f(0) < \mu c$ with constant history function on $[-\Delta, 0]$ with $q_1 = 4.99$ and $q_2 = 5.01$, $N = 2, \lambda = 10$, $\mu = 1$, $\beta=2$, $c=10$, $p = 1$, $f(x) = 1/(1 + \exp(\theta x))$. The left two plots are queue length versus time with $\Delta = .6$ (Left) and $\Delta = 2$ (Right). The right two plots are phase plots of the queue length derivative with respect to time against queue length for $\Delta = .6$ (Left) and $\Delta = 2$ (Right).}
\label{fig1}
\end{figure}

\begin{figure}[ht!]
\hspace{-15mm} \includegraphics[scale=.33]{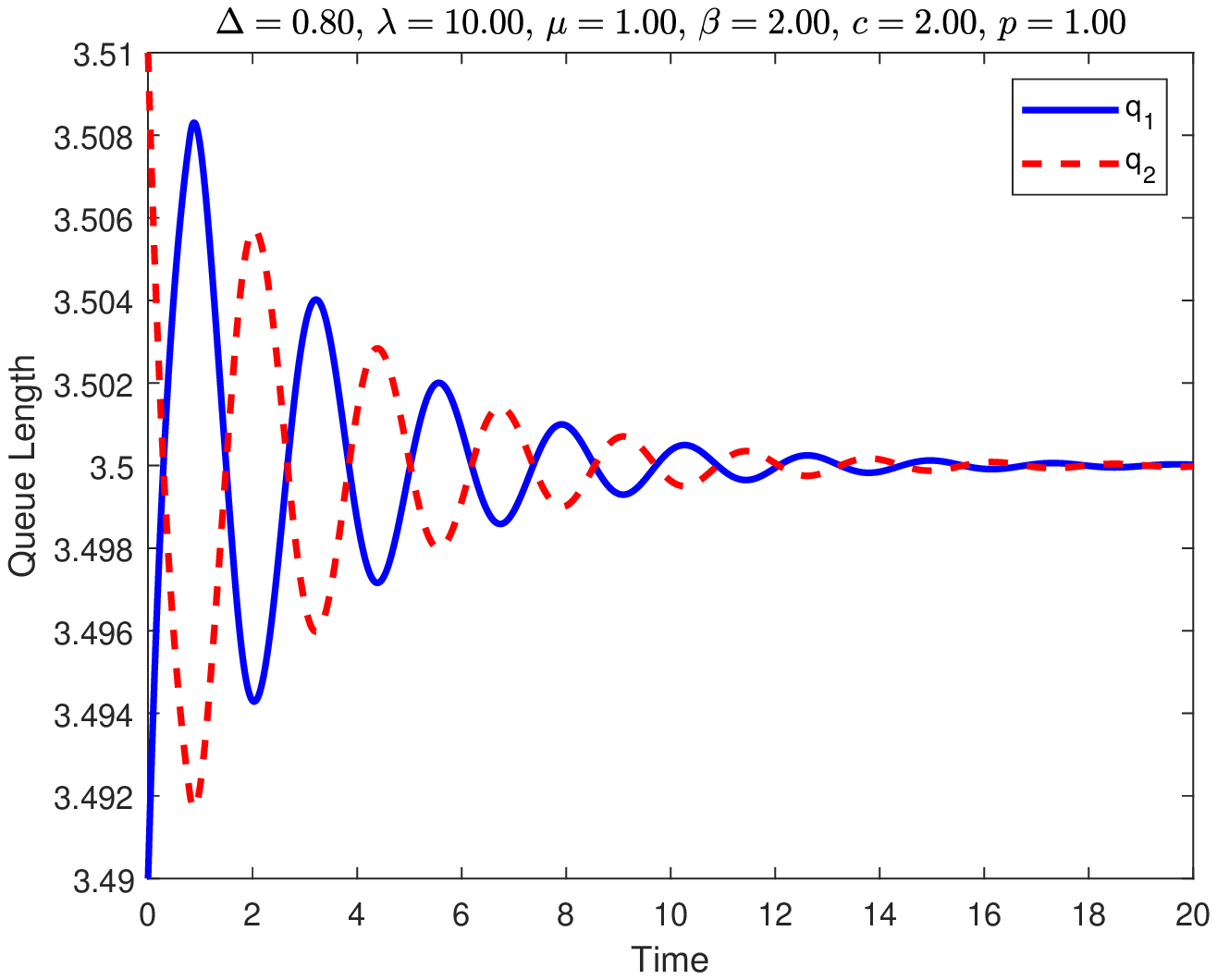}\includegraphics[scale=.33]{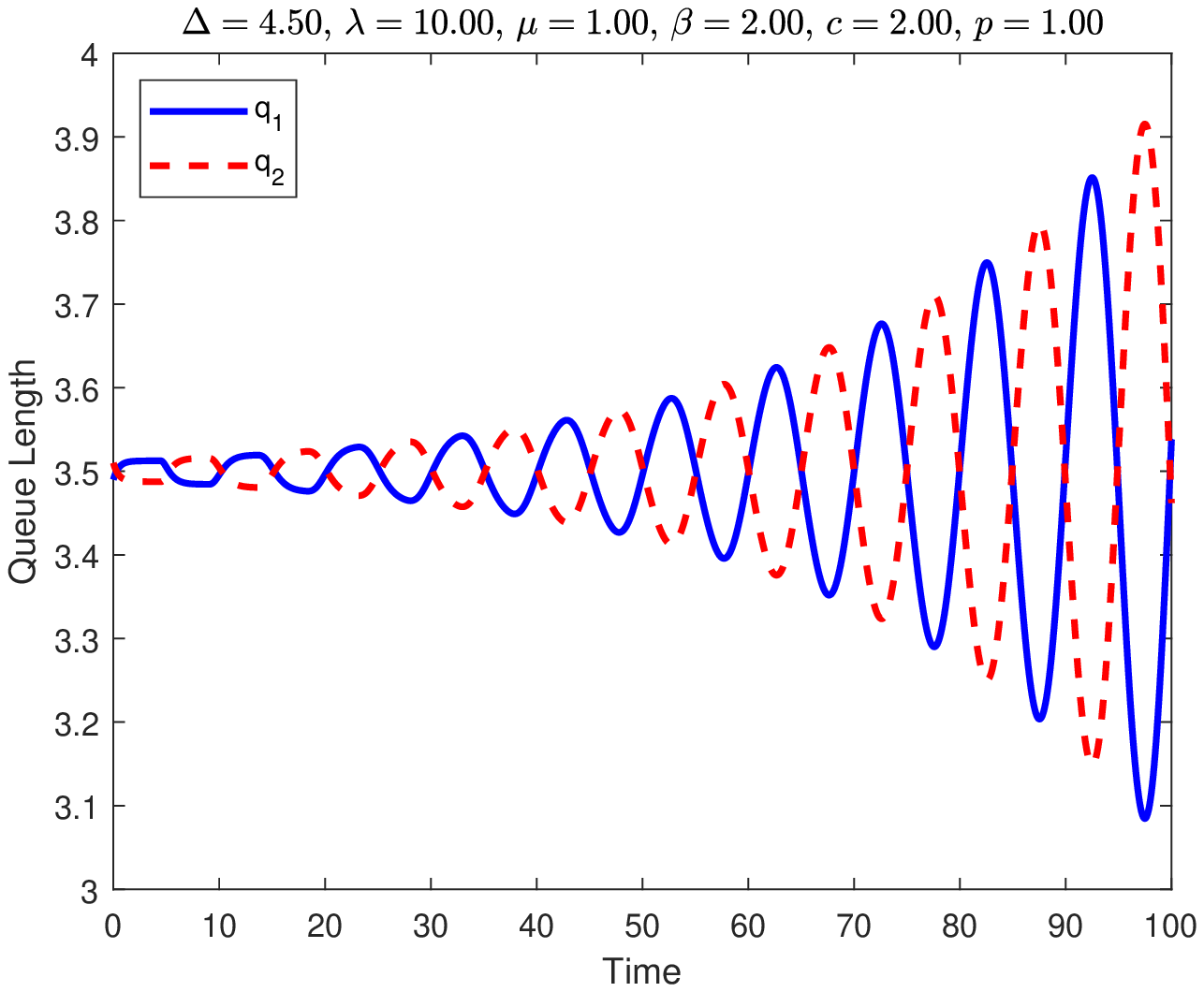}\includegraphics[scale=.33]{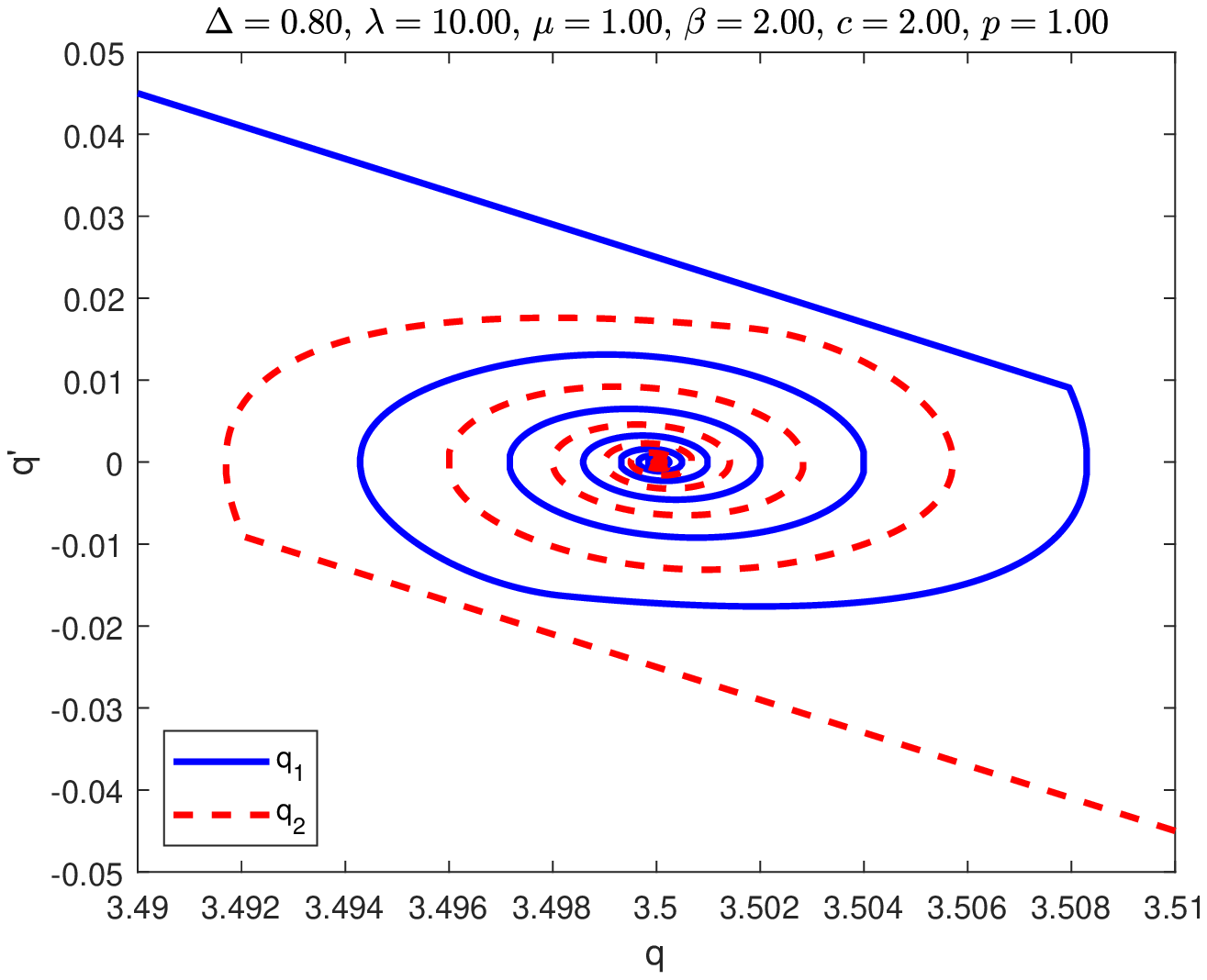}\includegraphics[scale=.33]{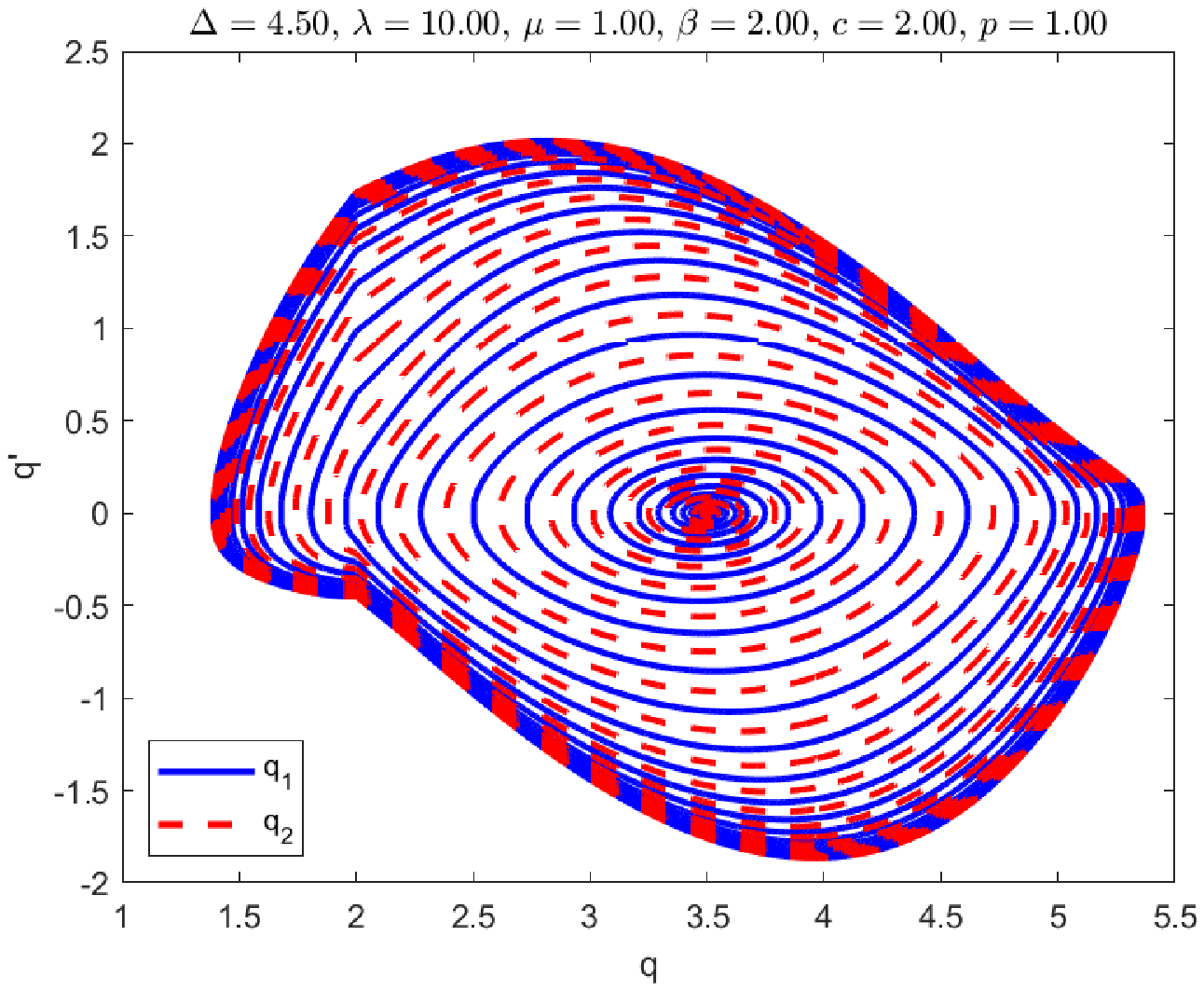}
\caption{Before and after the change in stability in the case where $\lambda f(0) > \mu c$ with constant history function on $[-\Delta, 0]$ with $q_1 = 4.99$ and $q_2 = 5.01$, $N = 2, \lambda = 10$, $\mu = 1$, $\beta=2$, $c=2$, $p = 1$, $f(x) = 1/(1 + \exp(\theta x))$. The left two plots are queue length versus time with $\Delta = .8$ (Left) and $\Delta = 4.5$ (Right). The right two plots are phase plots of the queue length derivative with respect to time against queue length for $\Delta = .8$ (Left) and $\Delta = 4.5$ (Right).}
\label{fig2}
\end{figure}

\begin{figure}[ht!]
\hspace{-15mm}\includegraphics[scale=.33]{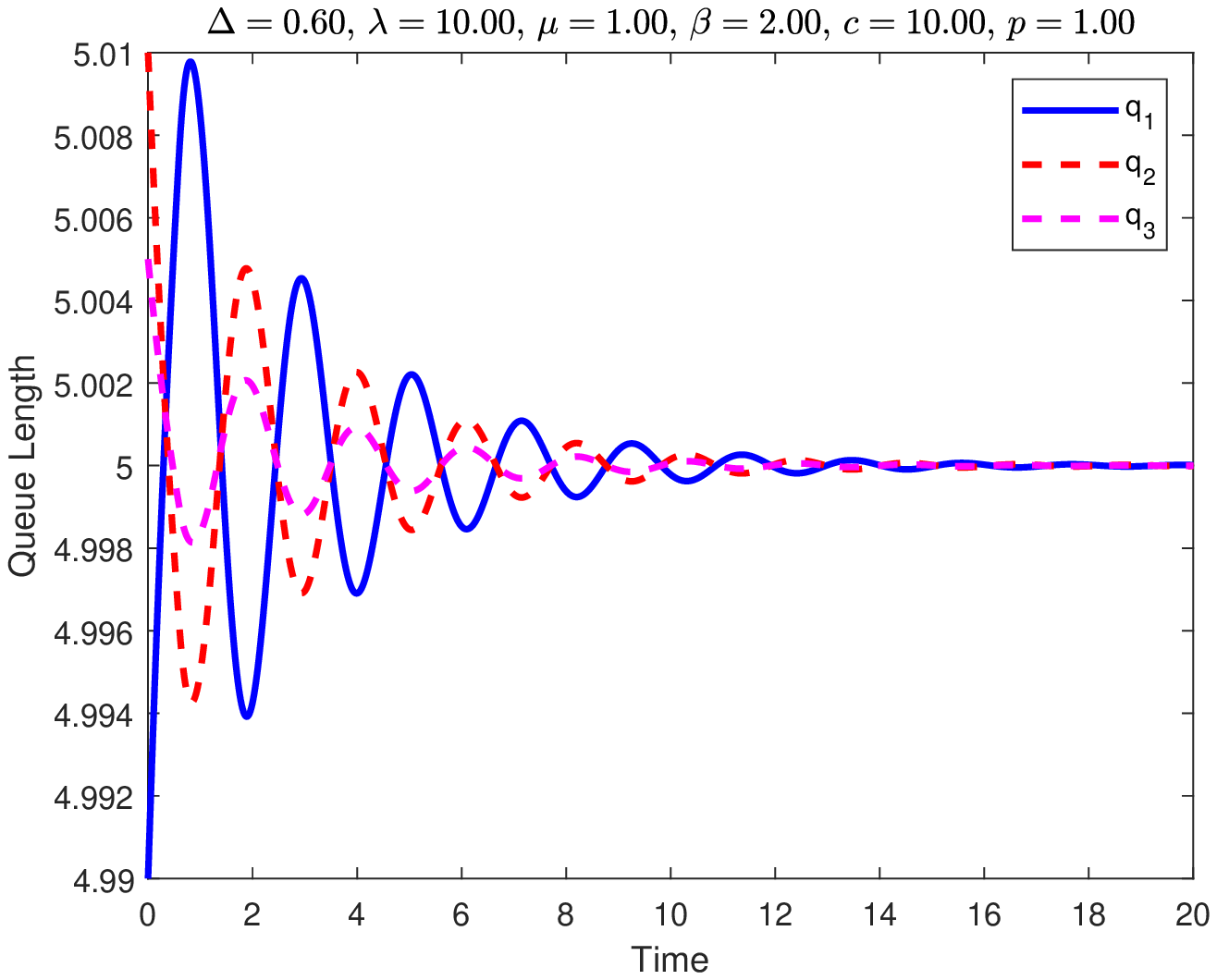}\includegraphics[scale=.33]{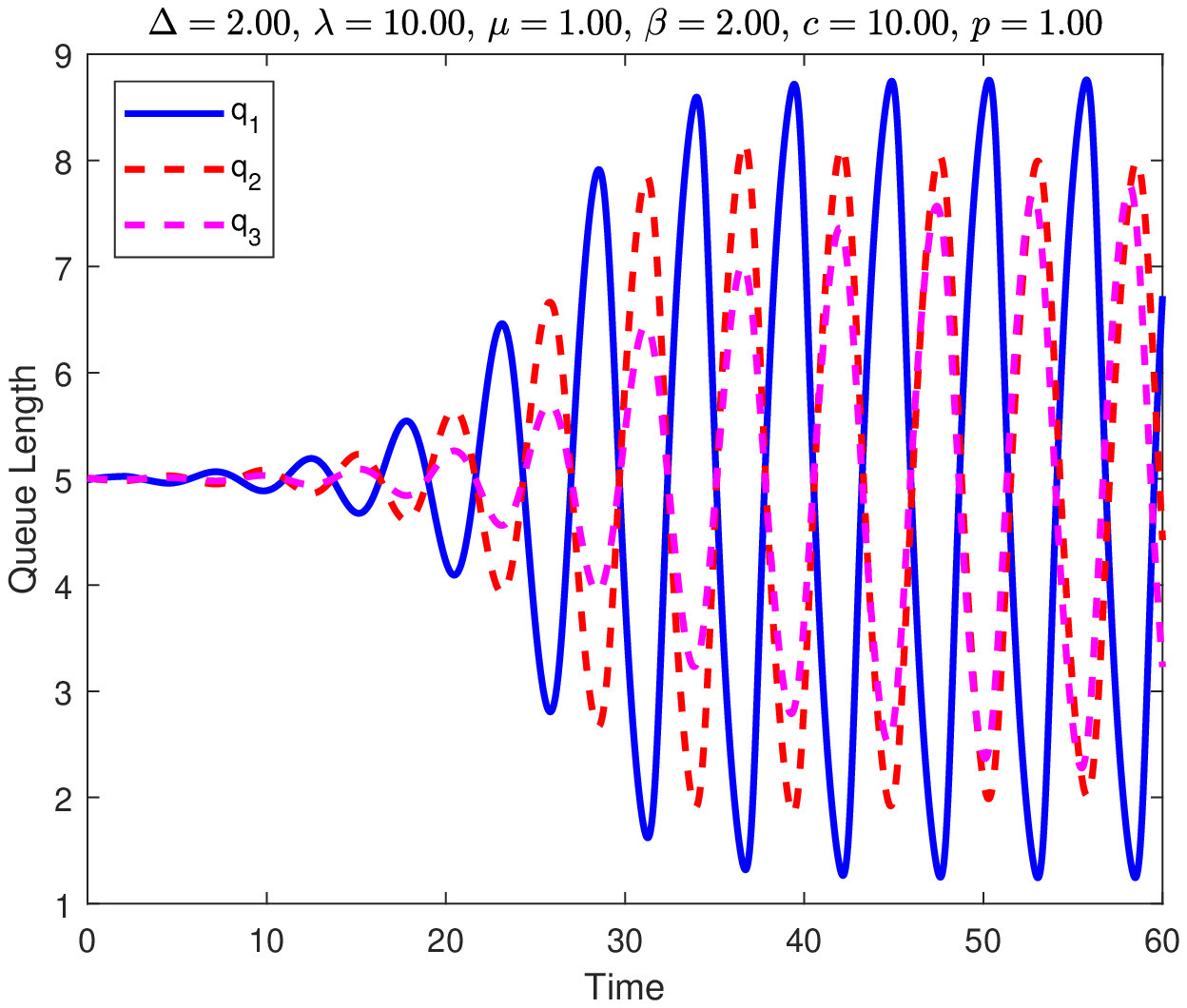}\includegraphics[scale=.33]{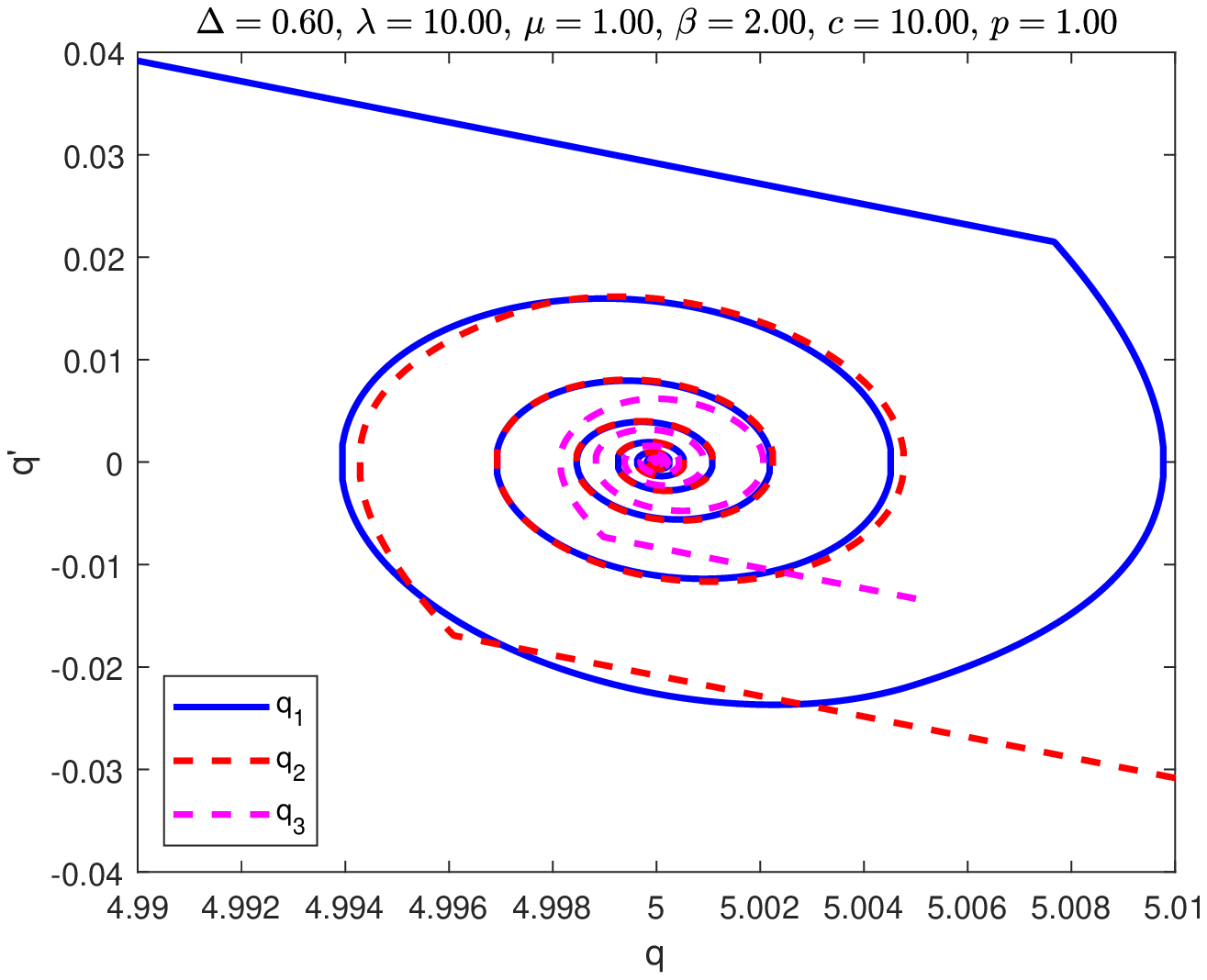}\includegraphics[scale=.33]{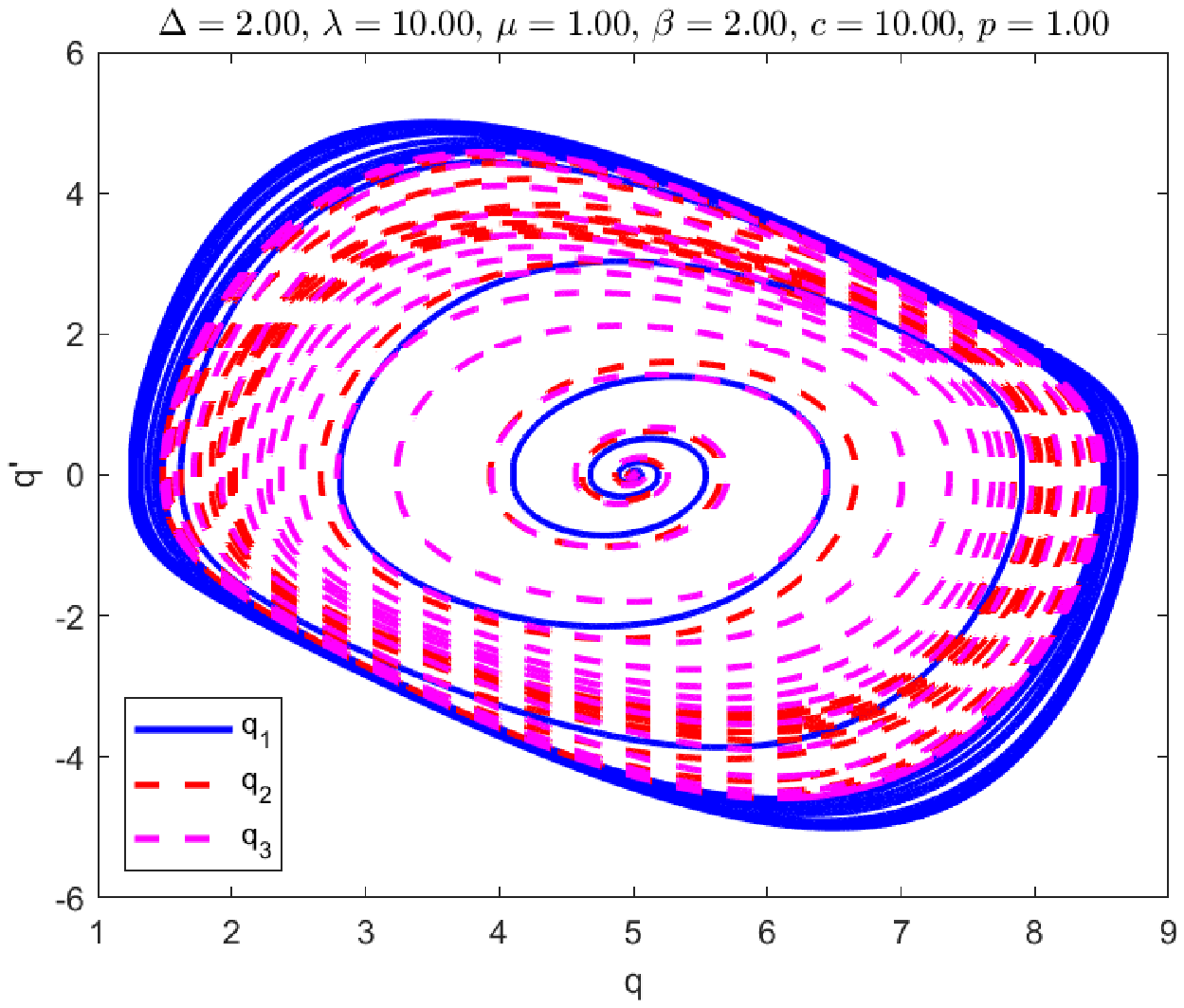}
\caption{Before and after the change in stability in the case where $\lambda f(0) < \mu c$ with constant history function on $[-\Delta, 0]$ with $q_1 = 4.99$, $q_2 = 5.01$, and $q_3 = 5.005$, $N = 3, \lambda = 10$, $\mu = 1$, $\beta=2$, $c=10$, $p = 1$, $f(x) = 1/(1 + \exp(\theta x))$. The left two plots are queue length versus time with $\Delta = .6$ (Left) and $\Delta = 2$ (Right). The right two plots are phase plots of the queue length derivative with respect to time against queue length for $\Delta = .6$ (Left) and $\Delta = 2$ (Right).}
\label{fig3}
\end{figure}

\begin{figure}[ht!]
\hspace{-15mm}\includegraphics[scale=.33]{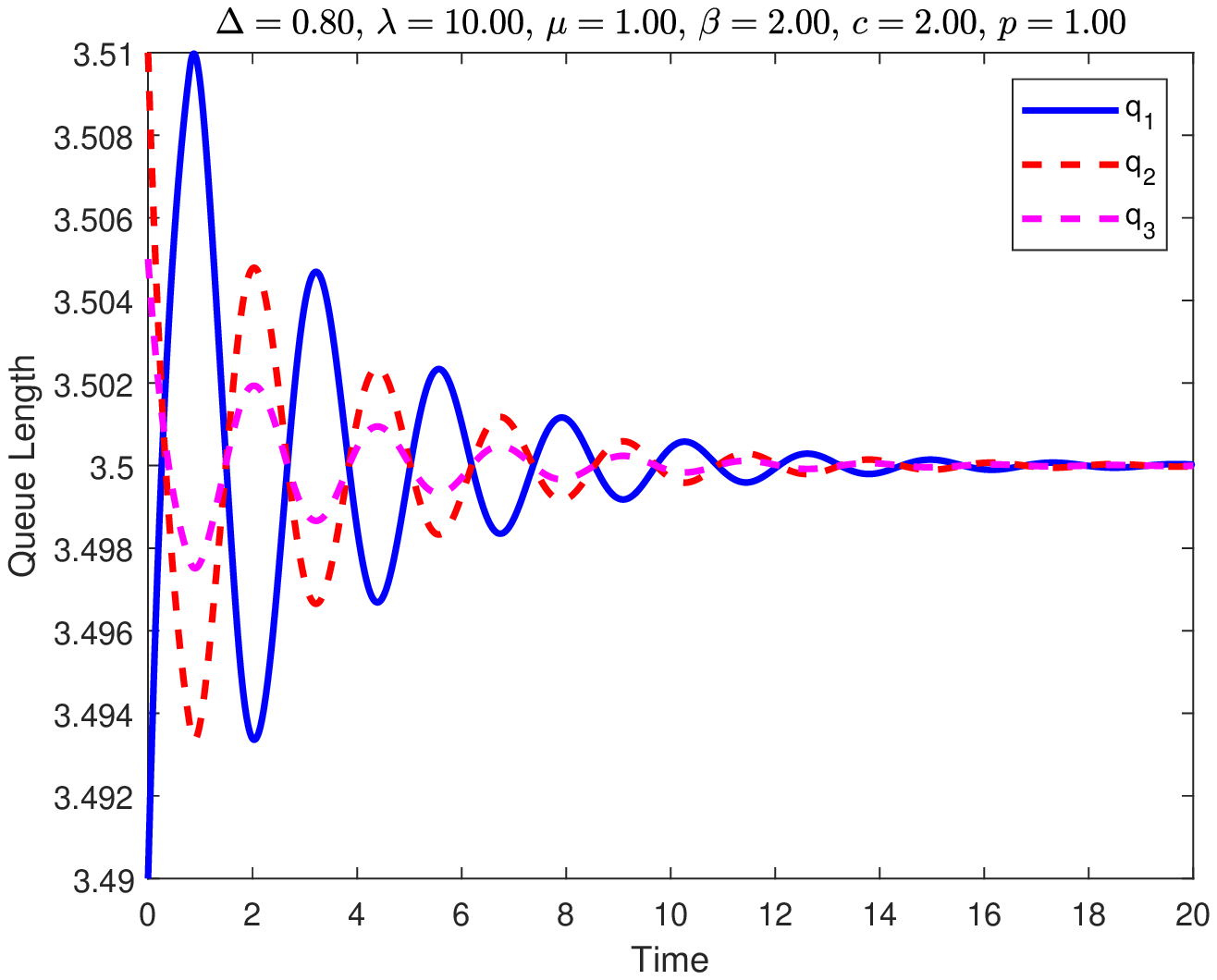}\includegraphics[scale=.33]{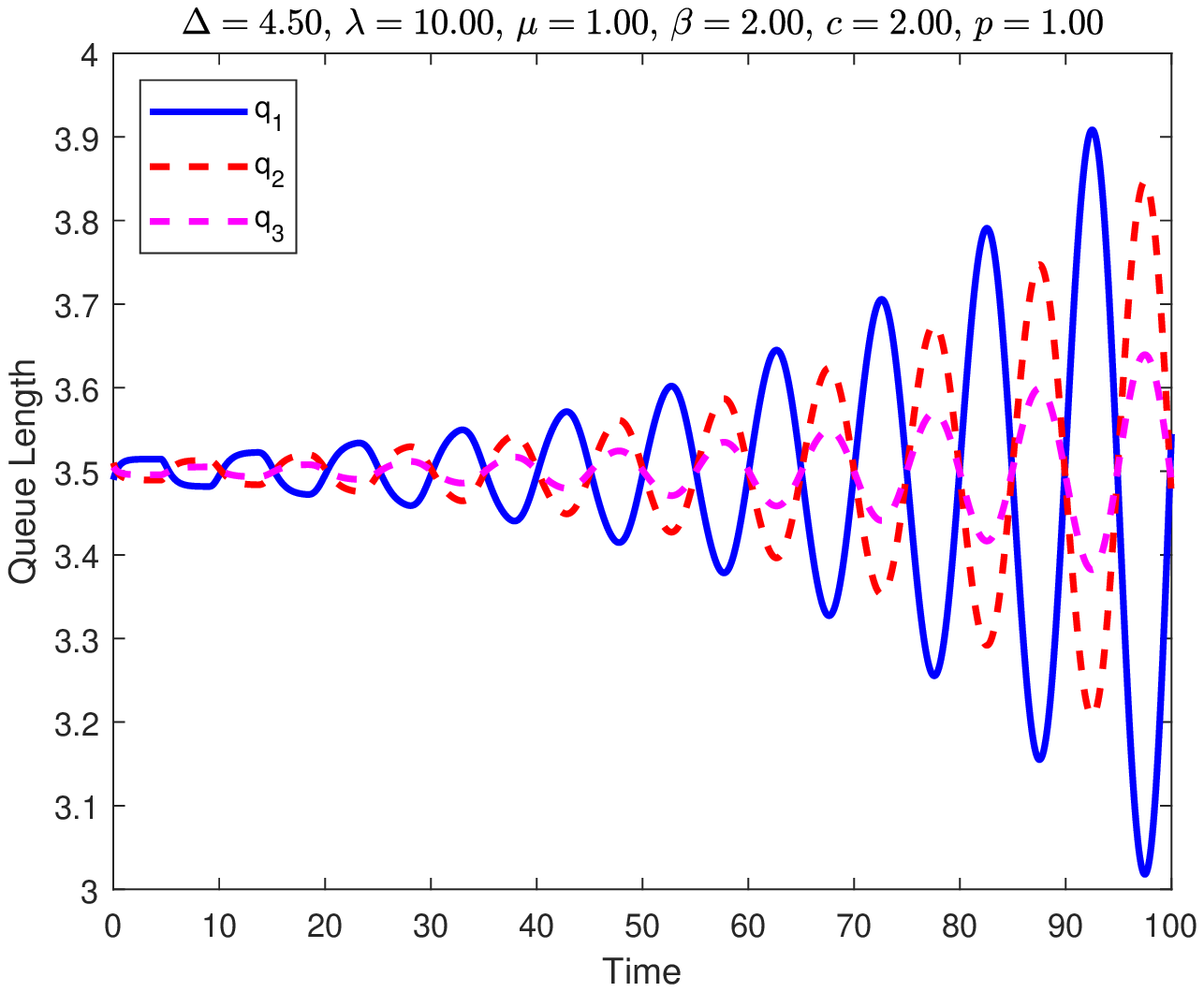}\includegraphics[scale=.33]{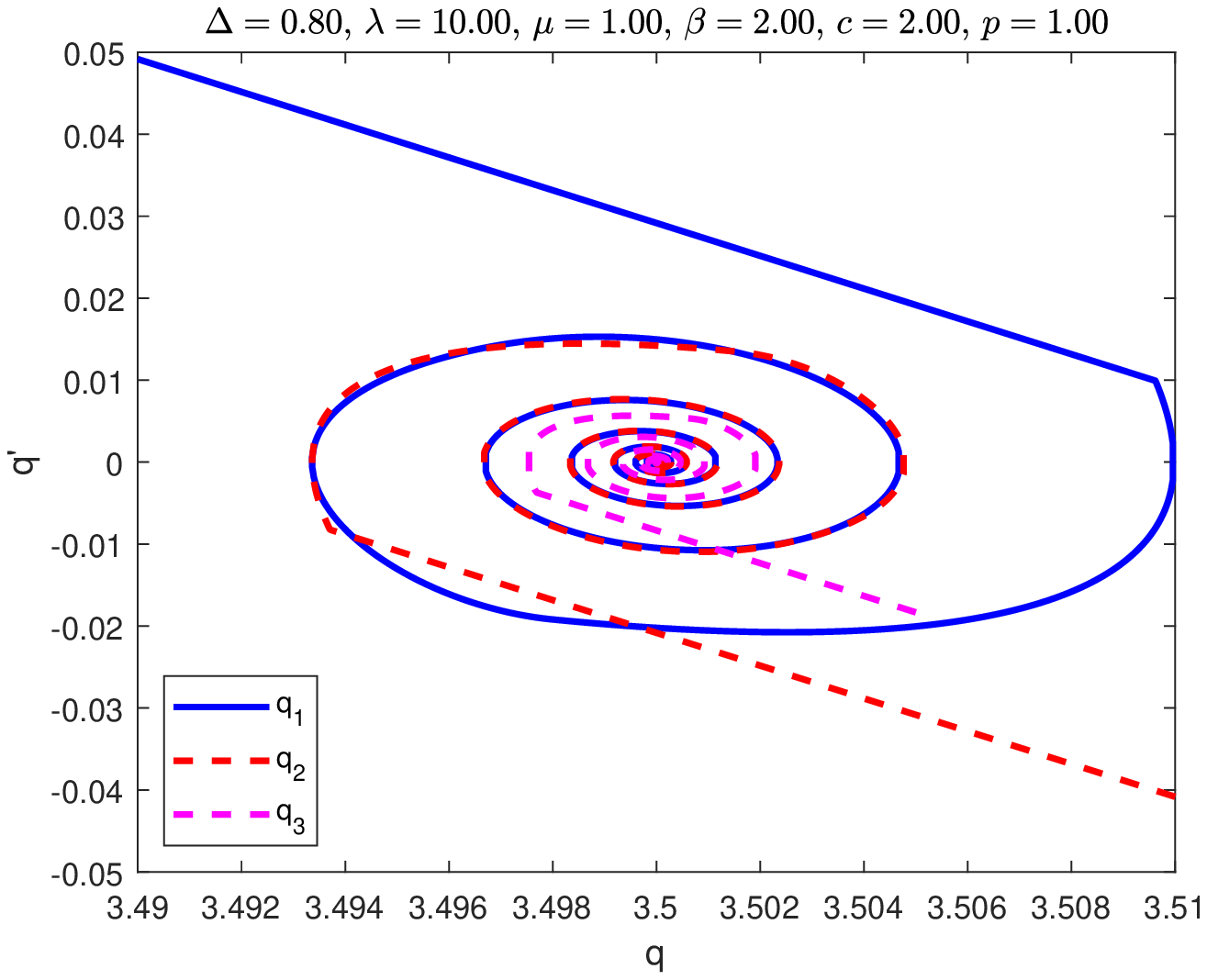}\includegraphics[scale=.33]{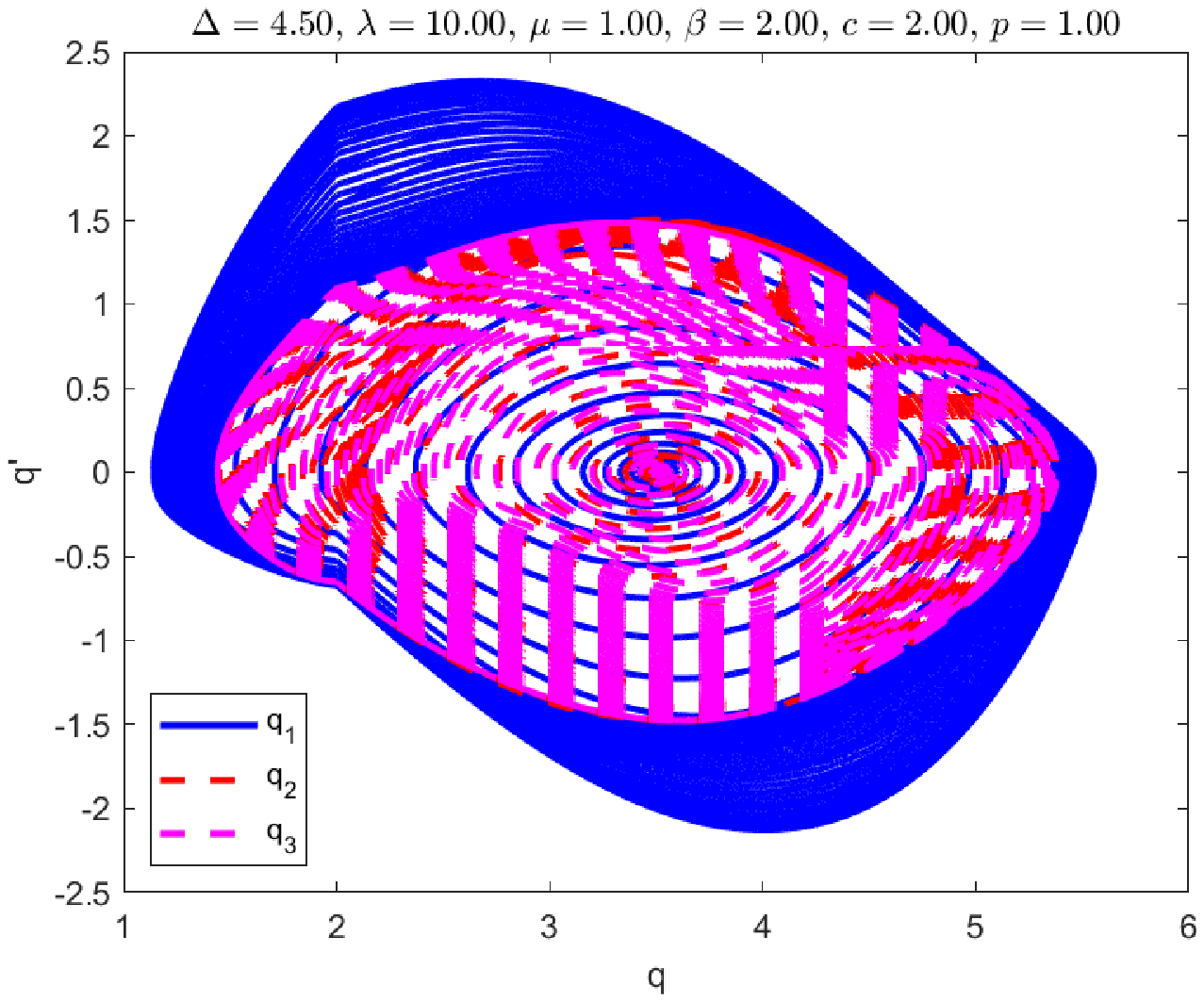}
\caption{Before and after the change in stability in the case where $\lambda f(0) > \mu c$ with constant history function on $[-\Delta, 0]$ with $q_1 = 4.99$, $q_2 = 5.01$, and $q_3 = 5.005$, $N = 3, \lambda = 10$, $\mu = 1$, $\beta=2$, $c=2$, $p = 1$, $f(x) = 1/(1 + \exp(\theta x))$. The left two plots are queue length versus time with $\Delta = .8$ (Left) and $\Delta = 4.5$ (Right). The right two plots are phase plots of the queue length derivative with respect to time against queue length for $\Delta = .8$ (Left) and $\Delta = 4.5$ (Right).}
\label{fig4}
\end{figure}

\begin{figure}[ht!]
\hspace{-15mm}\includegraphics[scale=.45]{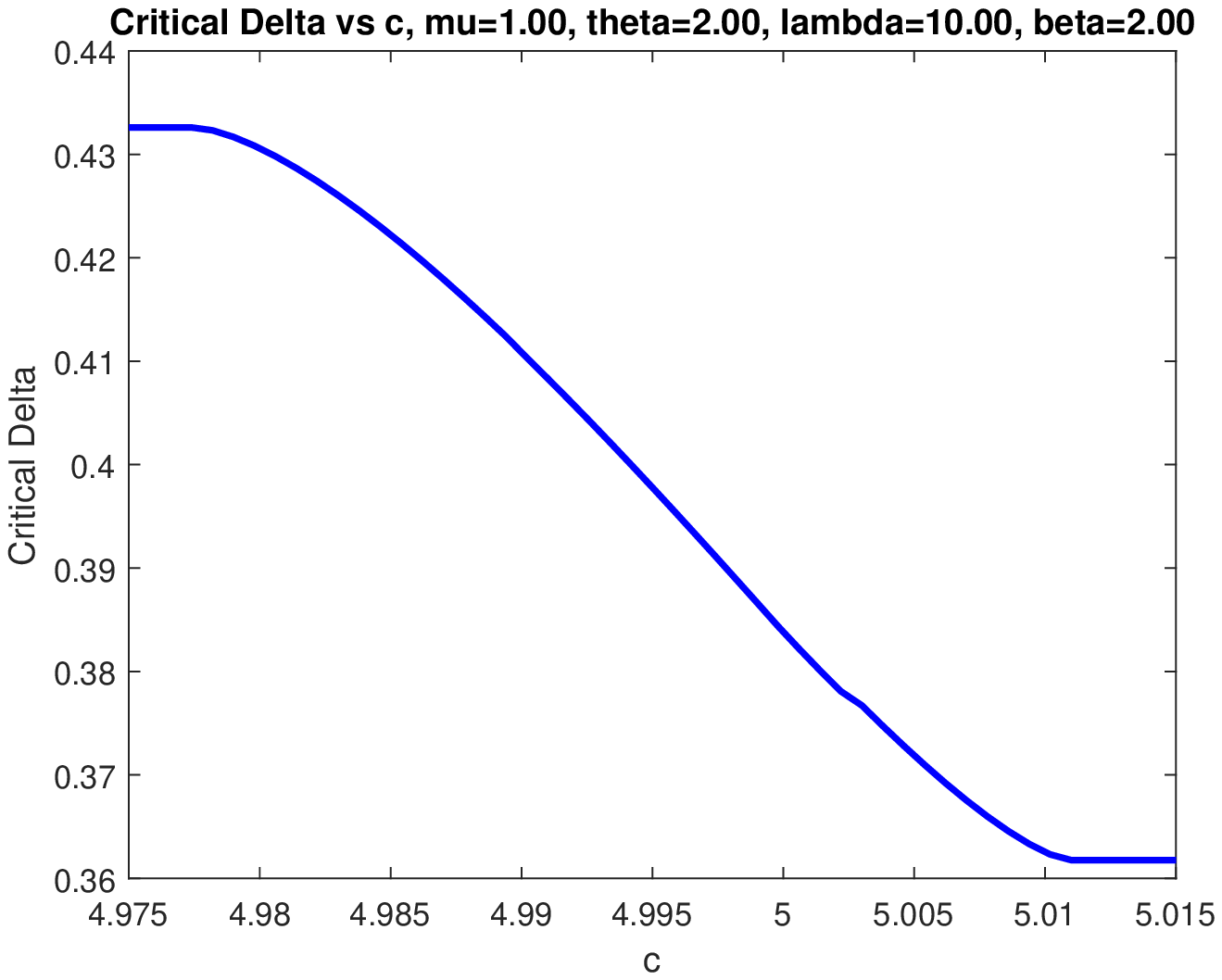}\includegraphics[scale=.45]{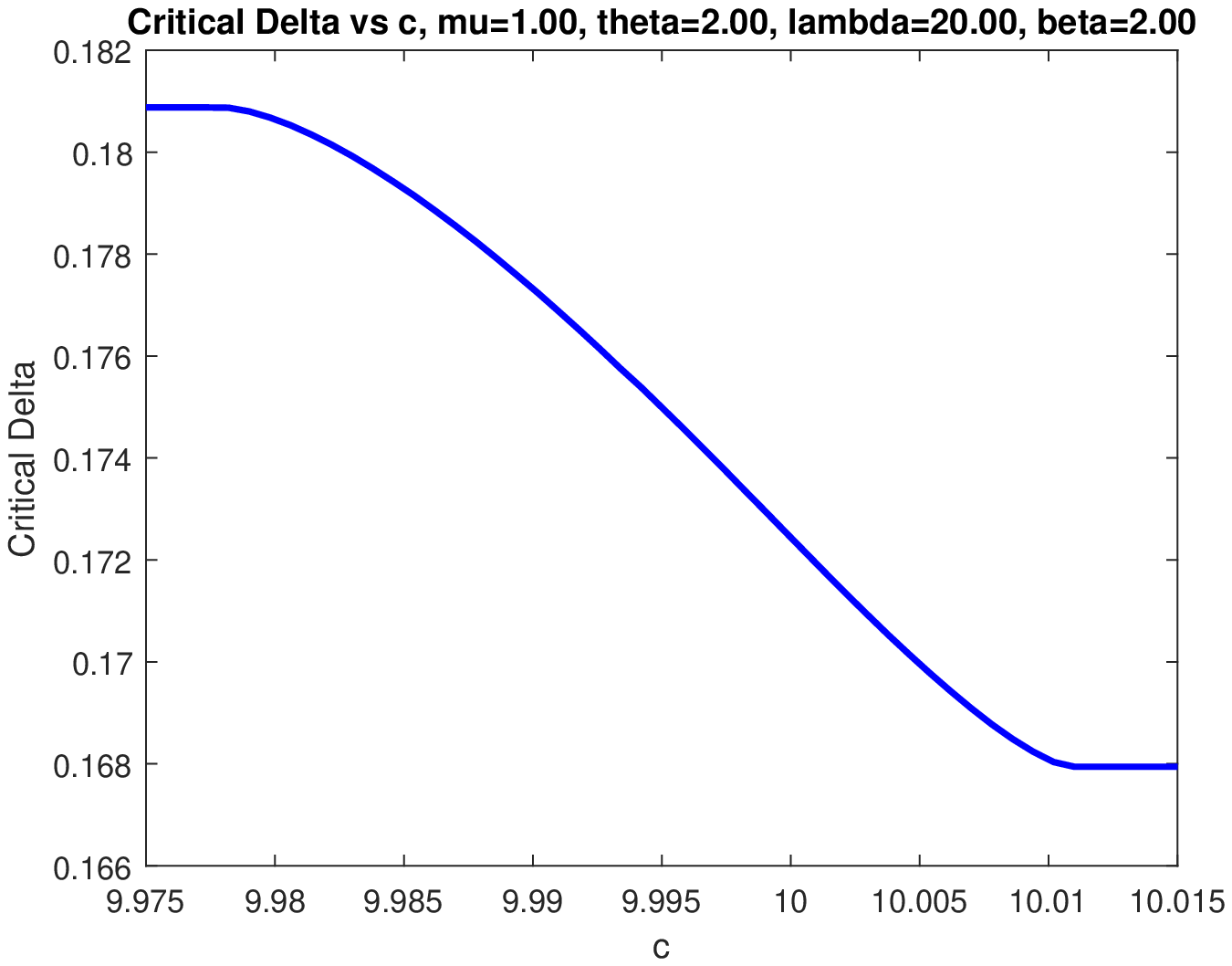}\includegraphics[scale=.45]{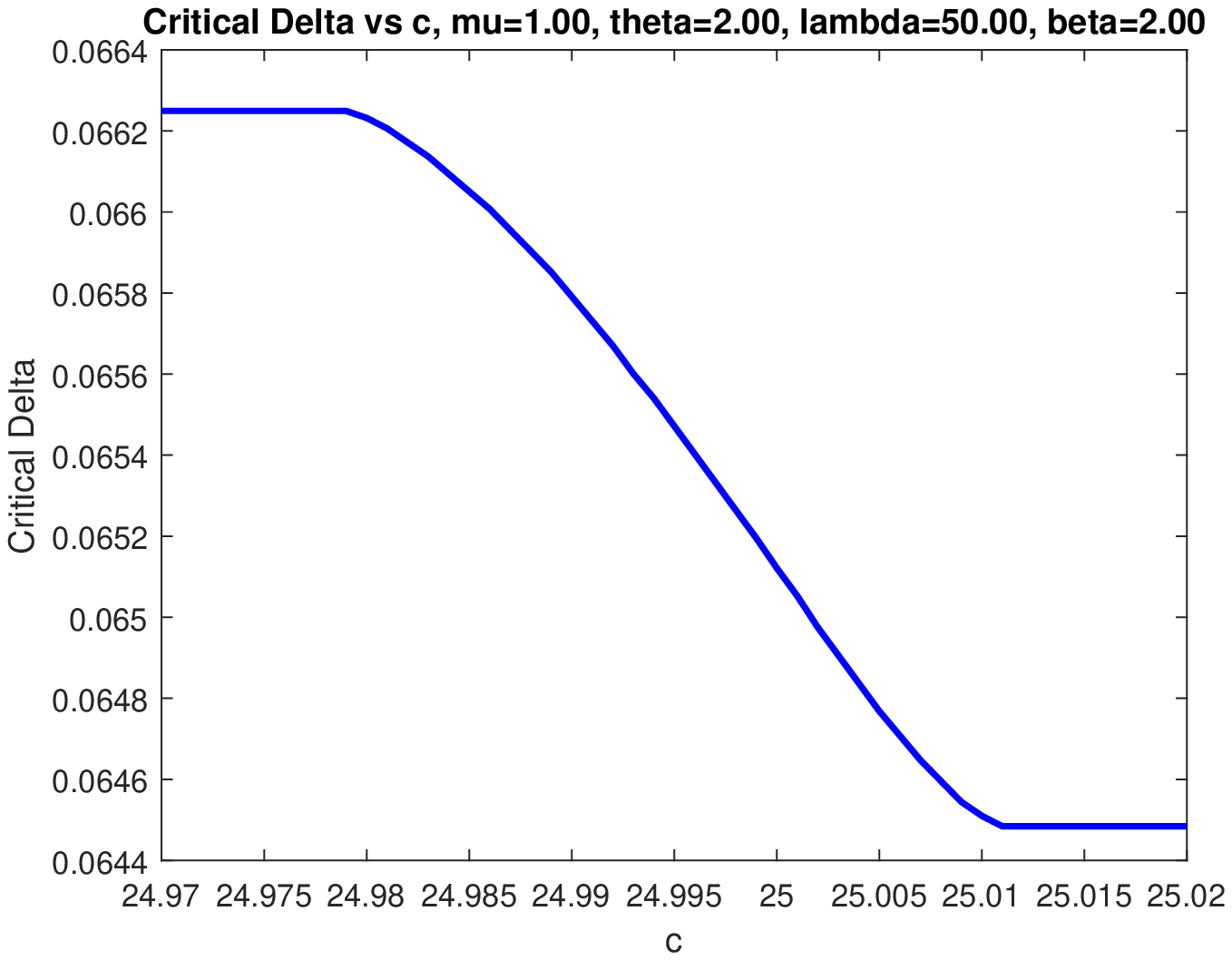}
\caption{How $\Delta_{\text{cr}}$ varies with $c$ and $f(x) = 1/(1 + \exp(\theta x))$ keeping the other parameters fixed with $\mu = 1, \beta = 2,$ and $\theta = 2$ for $\lambda = 10$ (Left), $\lambda = 20$ (Middle), and $\lambda = 50$ (Right). The parameter region is crossed when $c = \frac{\lambda f(0)}{\mu}$. The values of $\Delta_{\text{cr}}$ corresponding to the leftmost and rightmost values of $c$ plotted are $\Delta_{\text{cr}} = 0.4326$ and $\Delta_{\text{cr}} = 0.3618$ (Left), $\Delta_{\text{cr}} = 0.1809$ and $\Delta_{\text{cr}} = 0.1679$ (Middle), and $\Delta_{\text{cr}} = 0.0662$ and $\Delta_{\text{cr}} = 0.0645$ (Right).}
\label{fig5}
\end{figure}



\begin{figure}[ht!]
\hspace{-15mm}\includegraphics[scale=.45]{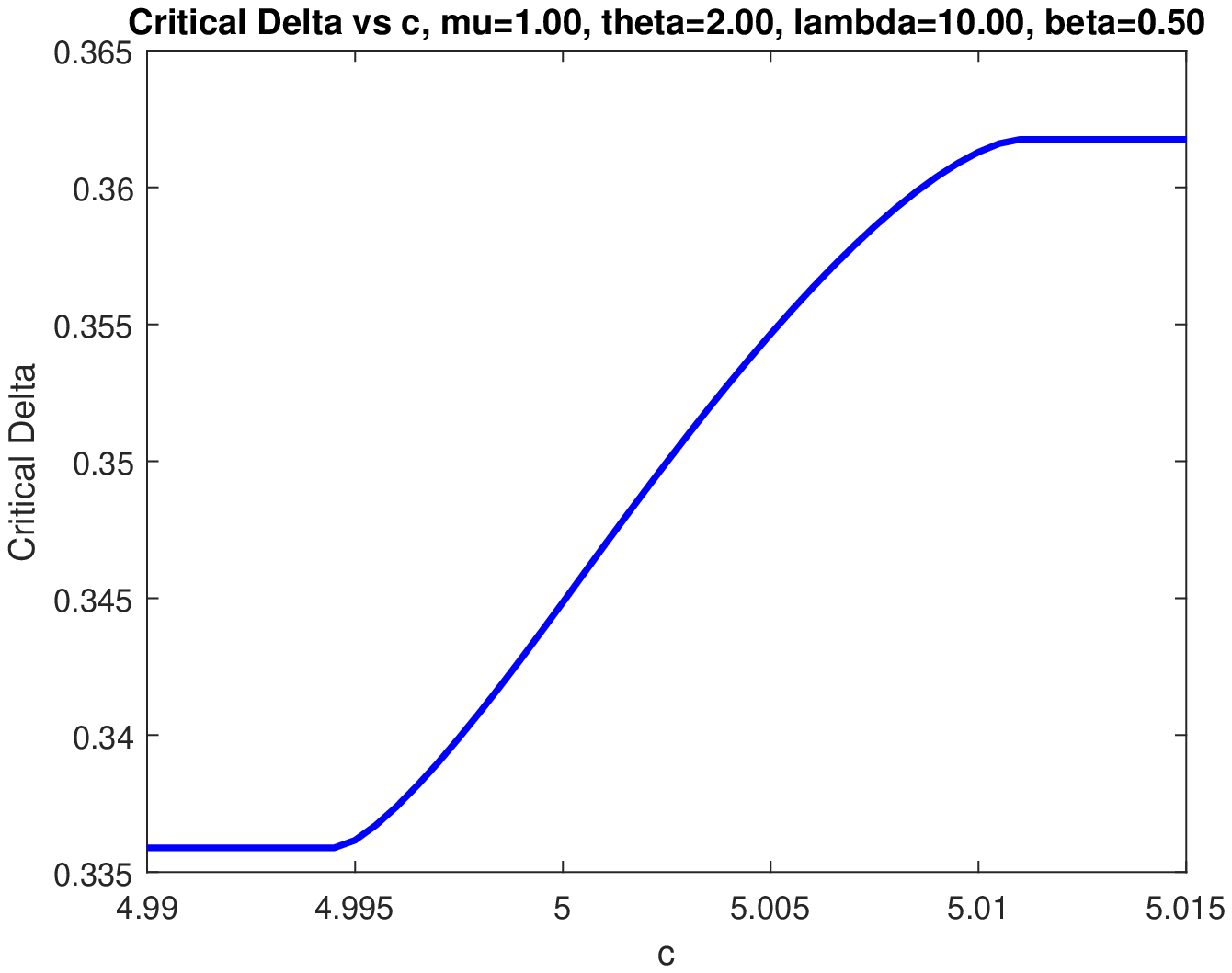}\includegraphics[scale=.45]{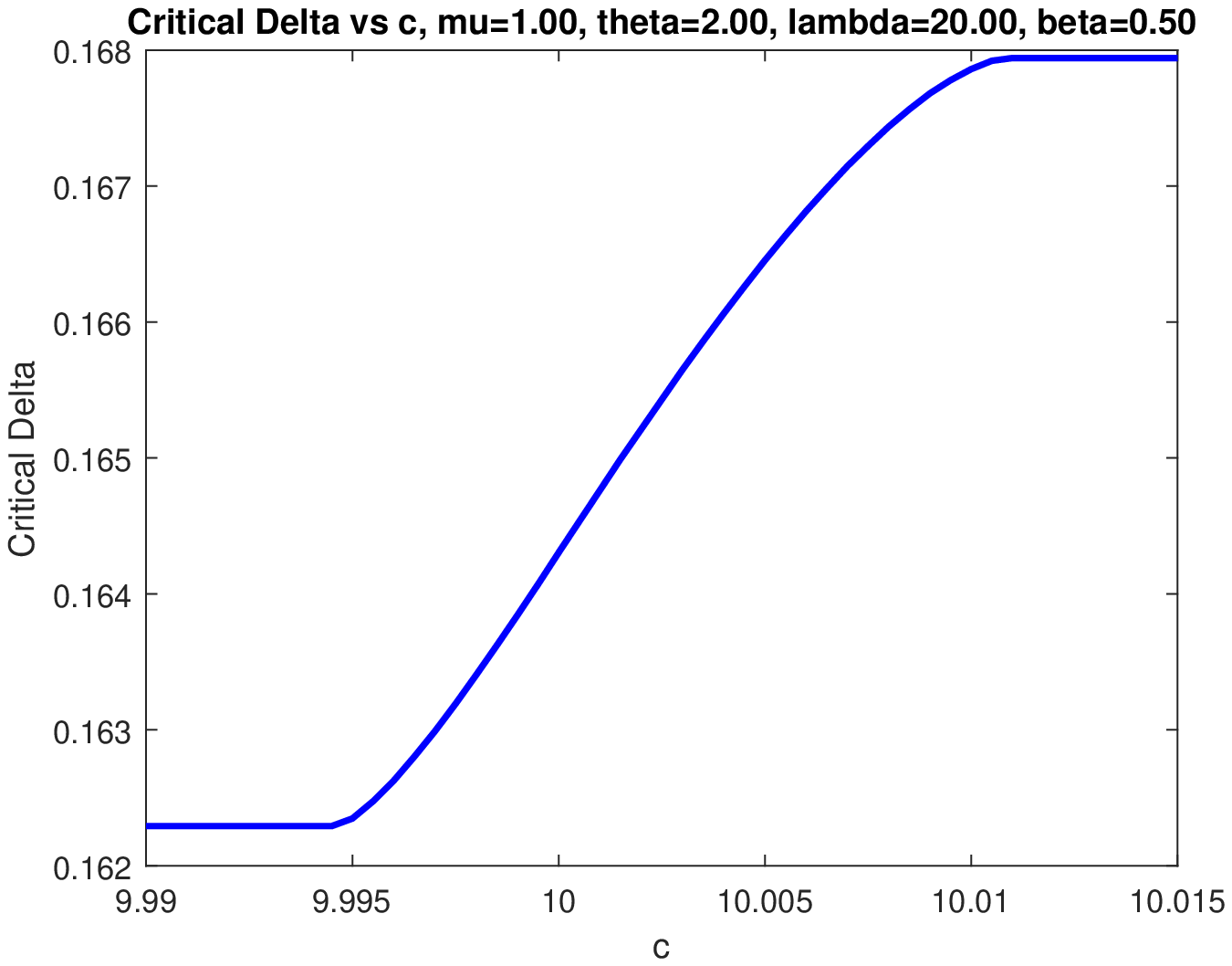}\includegraphics[scale=.45]{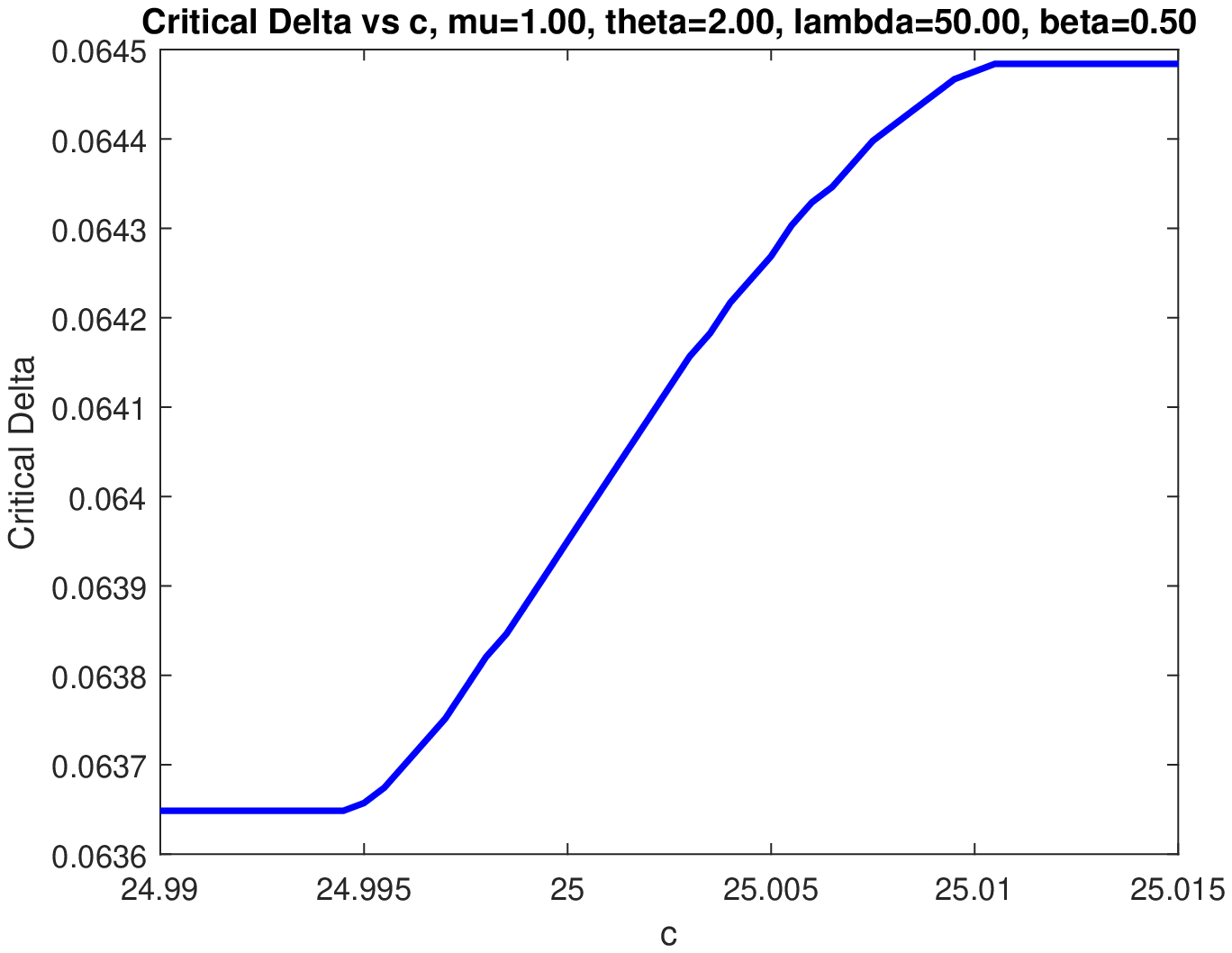}
\caption{How $\Delta_{\text{cr}}$ varies with $c$ and $f(x) = 1/(1 + \exp(\theta x))$ keeping the other parameters fixed with $\mu = 1, \beta = \frac{1}{2},$ and $\theta = 2$ for $\lambda = 10$ (Left), $\lambda = 20$ (Middle), and $\lambda = 50$ (Right). The parameter region is crossed when $c = \frac{\lambda f(0)}{\mu}$. The values of $\Delta_{\text{cr}}$ corresponding to the leftmost and rightmost values of $c$ plotted are $\Delta_{\text{cr}} = 0.3359$ and $\Delta_{\text{cr}} = 0.3618$ (Left), $\Delta_{\text{cr}} = 0.1623$ and $\Delta_{\text{cr}} = 0.1679$ (Middle), and $\Delta_{\text{cr}} = 0.0636$ and $\Delta_{\text{cr}} = 0.0645$ (Right).}
\label{fig6}
\end{figure}


\begin{figure}[ht!]
\hspace{-15mm}\includegraphics[scale=.33]{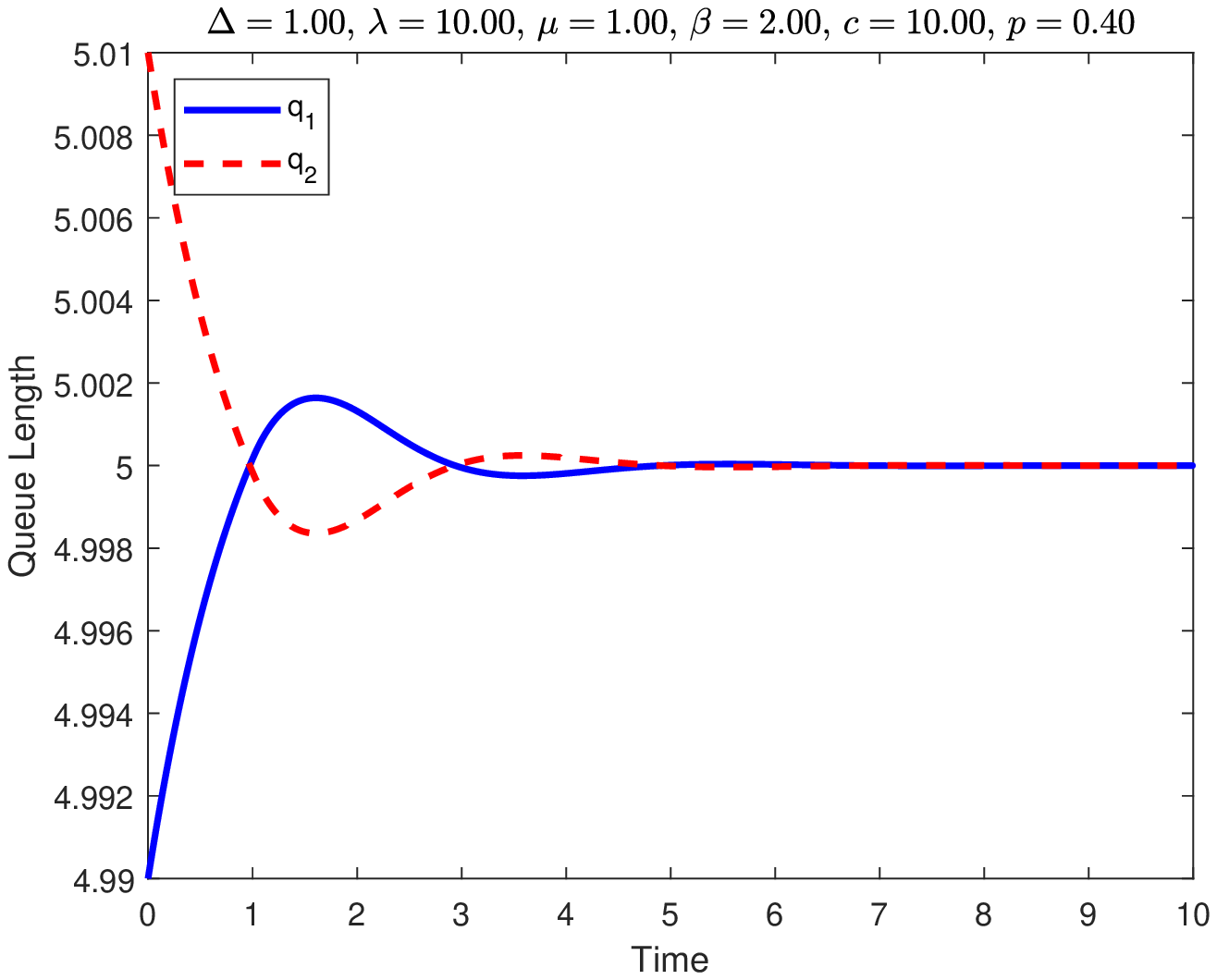}\includegraphics[scale=.33]{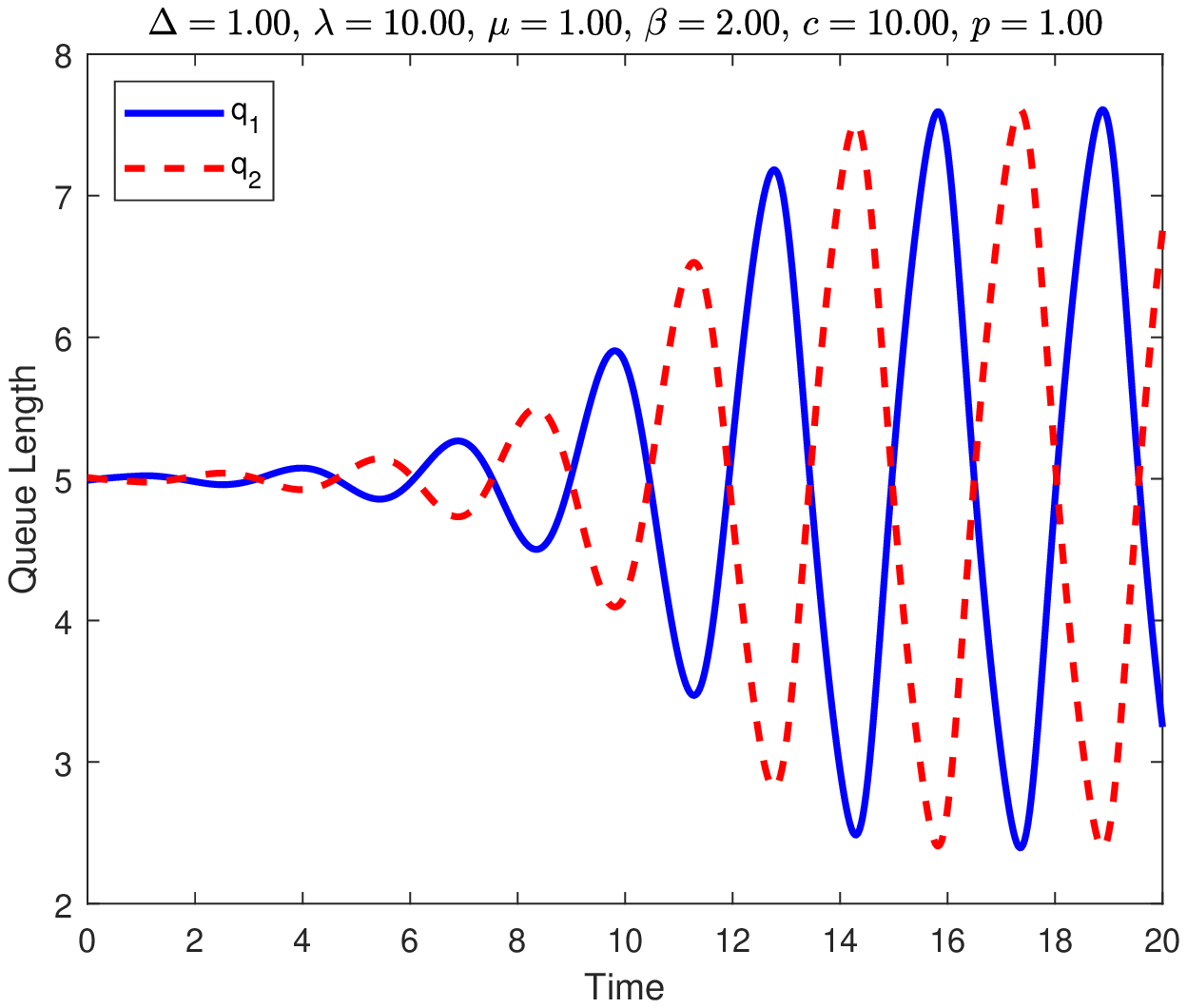}\includegraphics[scale=.33]{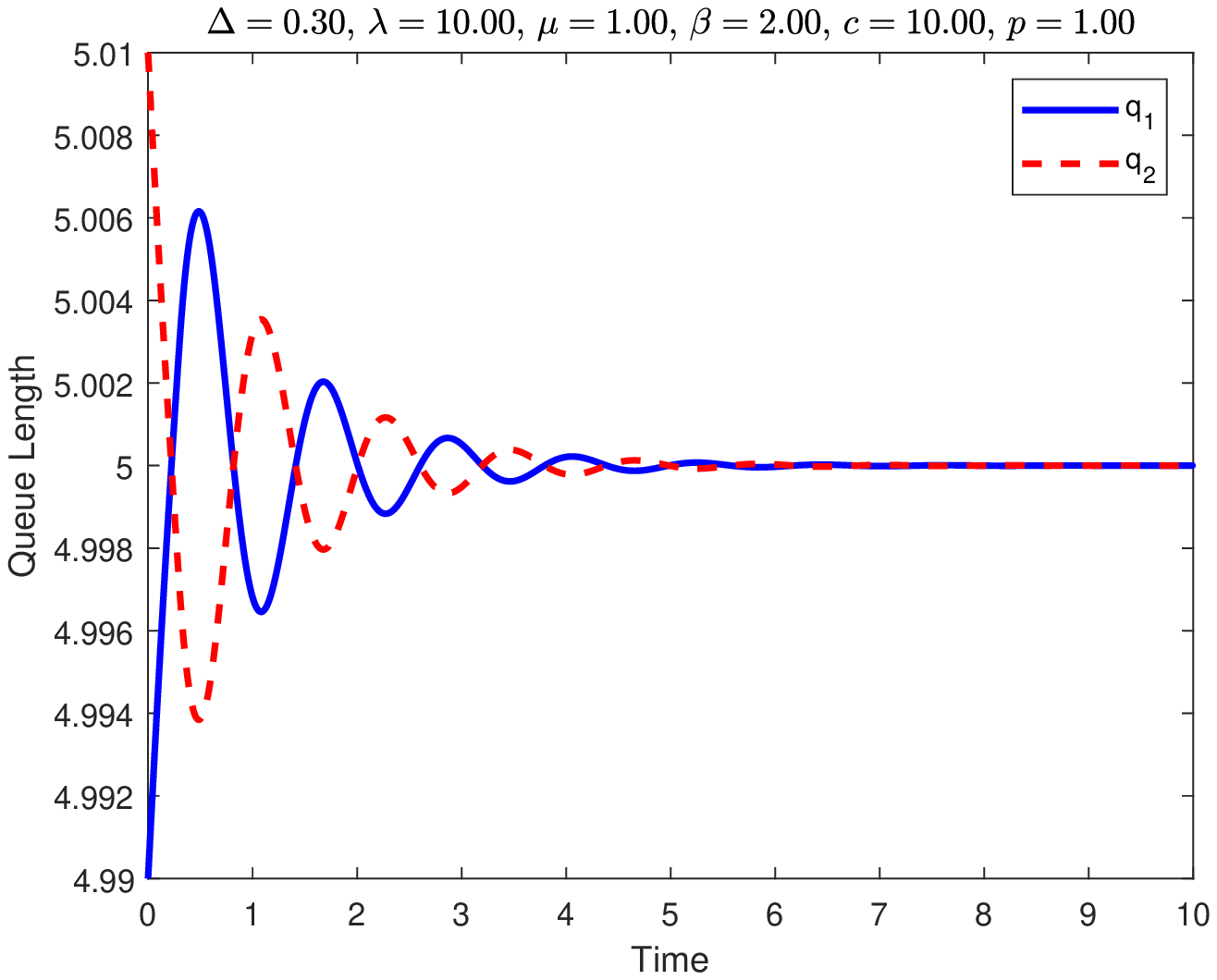}\includegraphics[scale=.33]{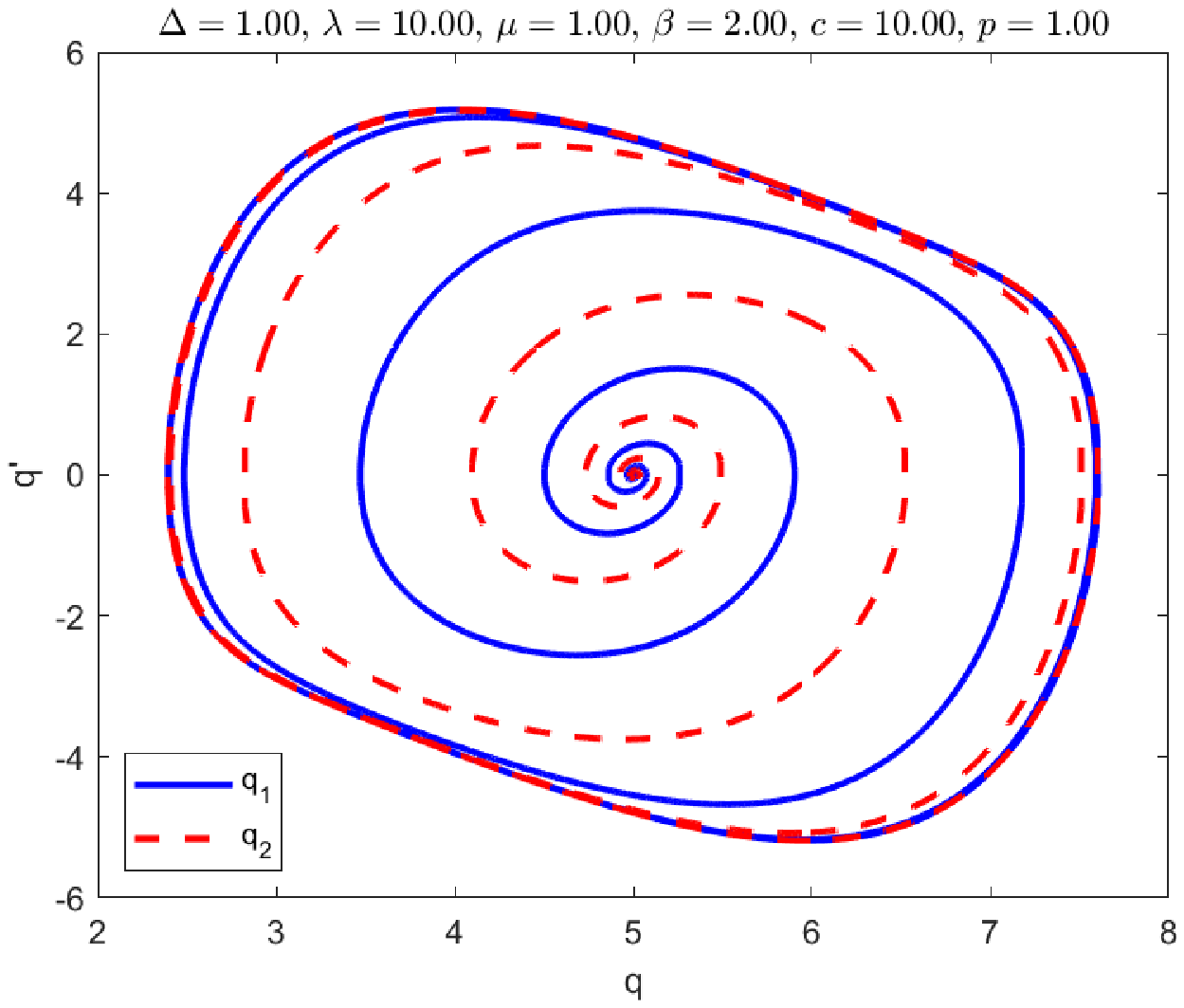}
\caption{Before and after the change in stability in the case where $\lambda f(0) < \mu c$ with constant history function on $[-\Delta, 0]$ with $q_1 = 4.99$ and $q_2 = 5.01$, $N = 2, \lambda = 10$, $\mu = 1$, $\beta=2$, $c=10$, $p = 1$, $f(x) = 1 - \int^{x}_{-\infty} \frac{1}{\sqrt{2\pi}} e^{-y^2/2}dy$. The left two plots are queue length versus time with $\Delta = .3$ (Left) and $\Delta = 1$ (Right). The right two plots are phase plots of the queue length derivative with respect to time against queue length for $\Delta = .3$ (Left) and $\Delta = 1$ (Right).}
\label{fig7}
\end{figure}

\begin{figure}[ht!]
\hspace{-15mm}\includegraphics[scale=.33]{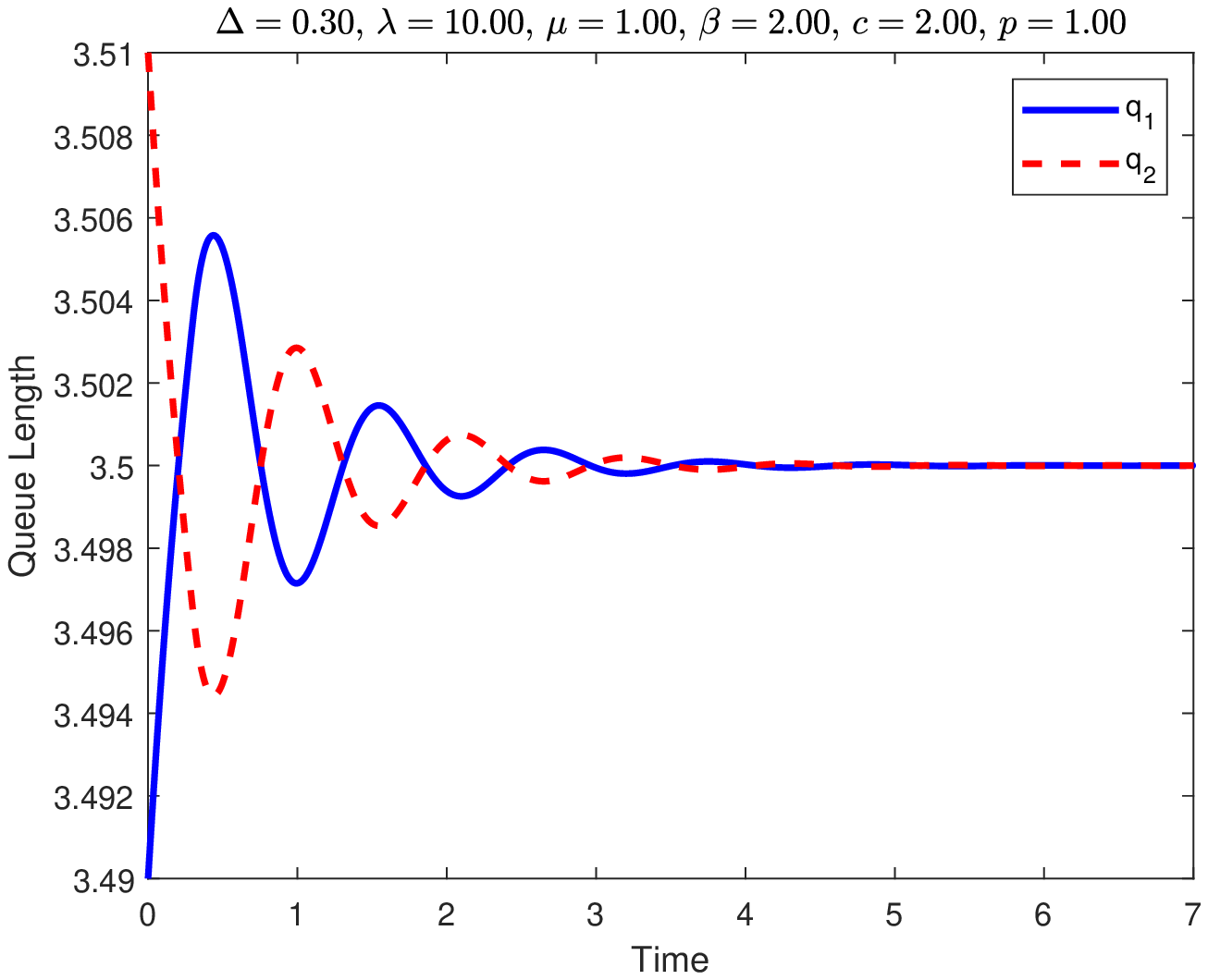}\includegraphics[scale=.33]{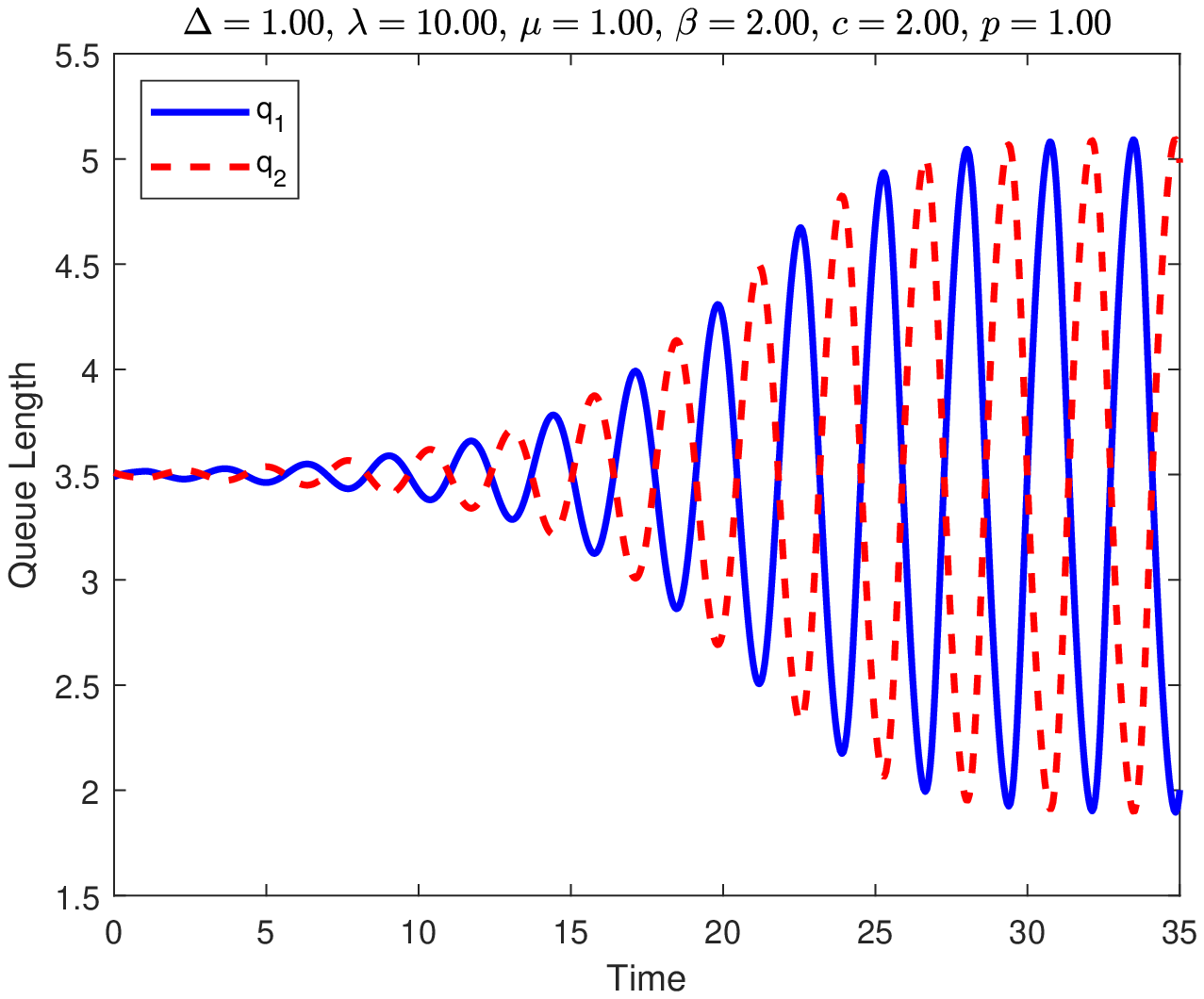}\includegraphics[scale=.33]{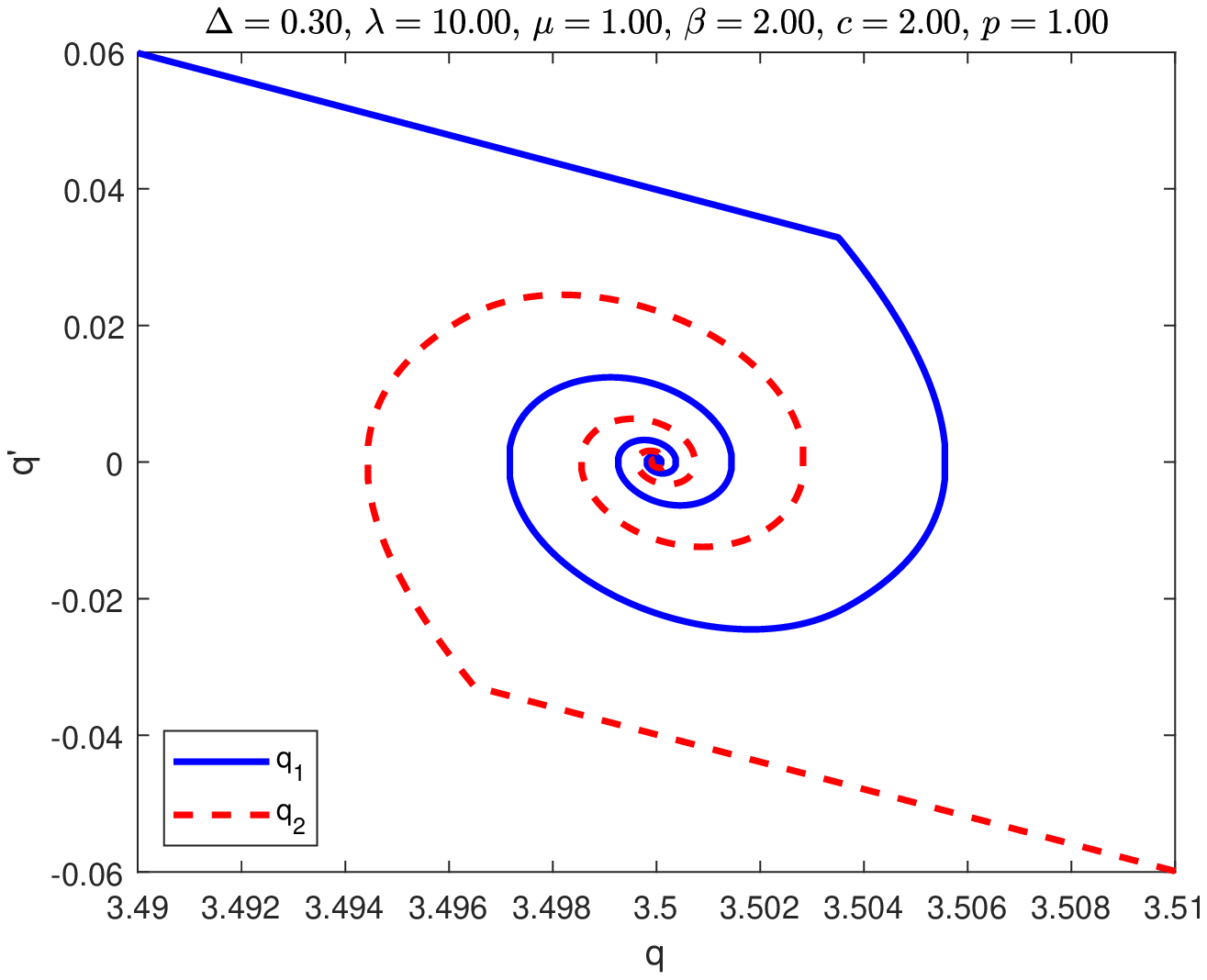}\includegraphics[scale=.33]{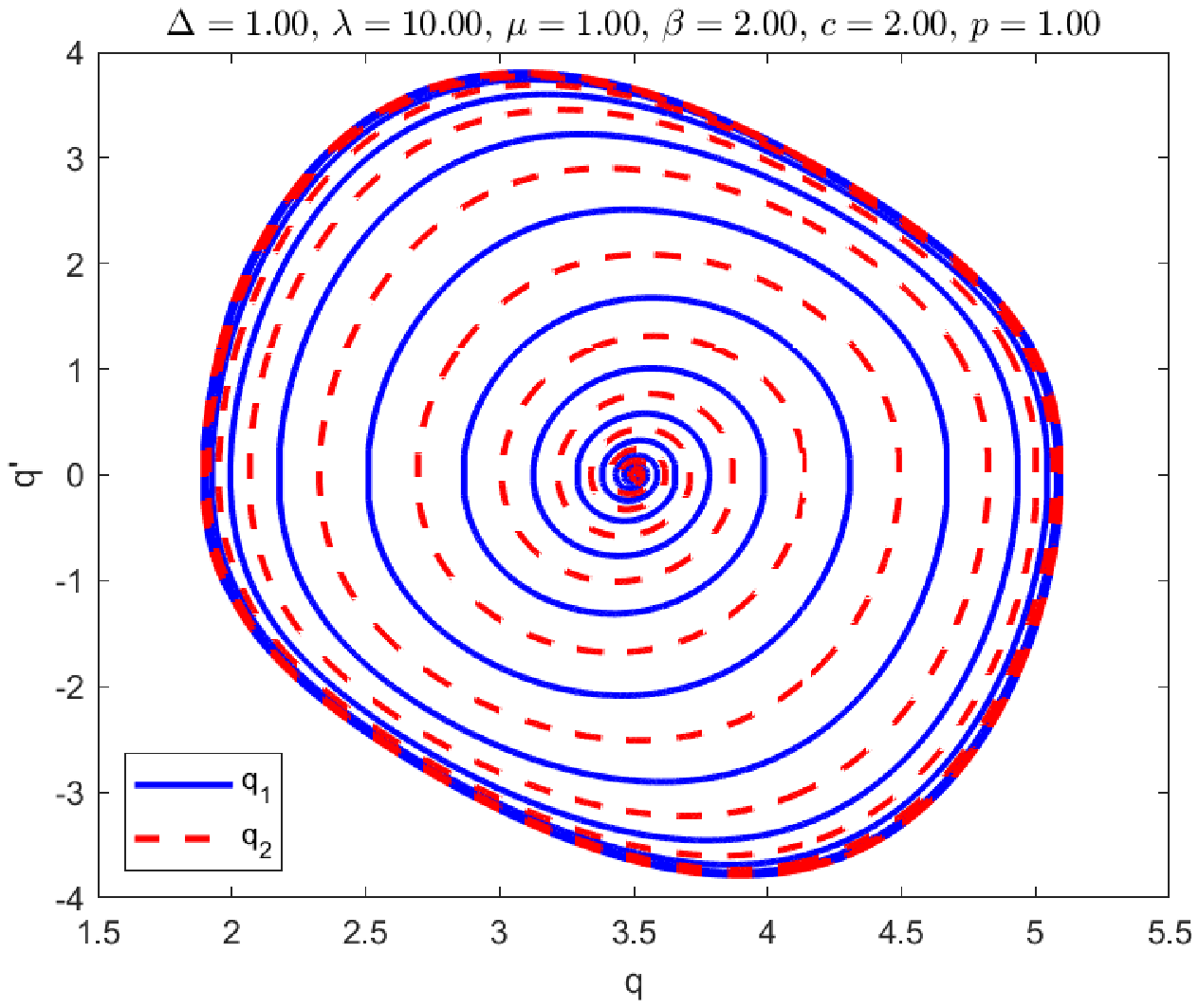}
\caption{Before and after the change in stability in the case where $\lambda f(0) > \mu c$ with constant history function on $[-\Delta, 0]$ with $q_1 = 4.99$ and $q_2 = 5.01$, $N = 2, \lambda = 10$, $\mu = 1$, $\beta=2$, $c=2$, $p = 1$, $f(x) =  \int^{\infty}_{x} \frac{1}{\sqrt{2\pi}} e^{-y^2/2}dy$. The left two plots are queue length versus time with $\Delta = .3$ (Left) and $\Delta = 1$ (Right). The right two plots are phase plots of the queue length derivative with respect to time against queue length for $\Delta = .3$ (Left) and $\Delta = 1$ (Right).}
\label{fig8}
\end{figure}

\begin{figure}[ht!]
\hspace{-15mm}\includegraphics[scale=.33]{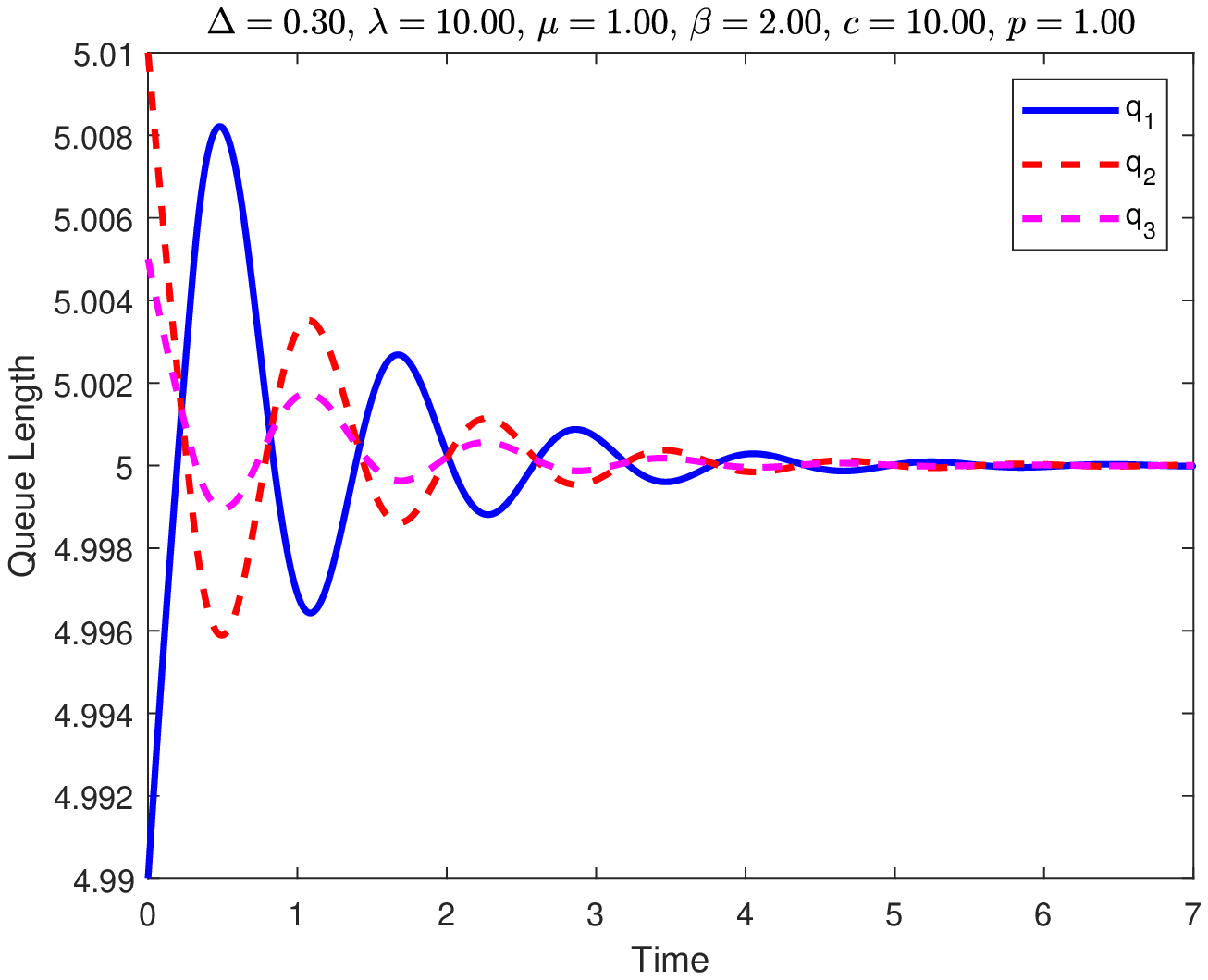}\includegraphics[scale=.33]{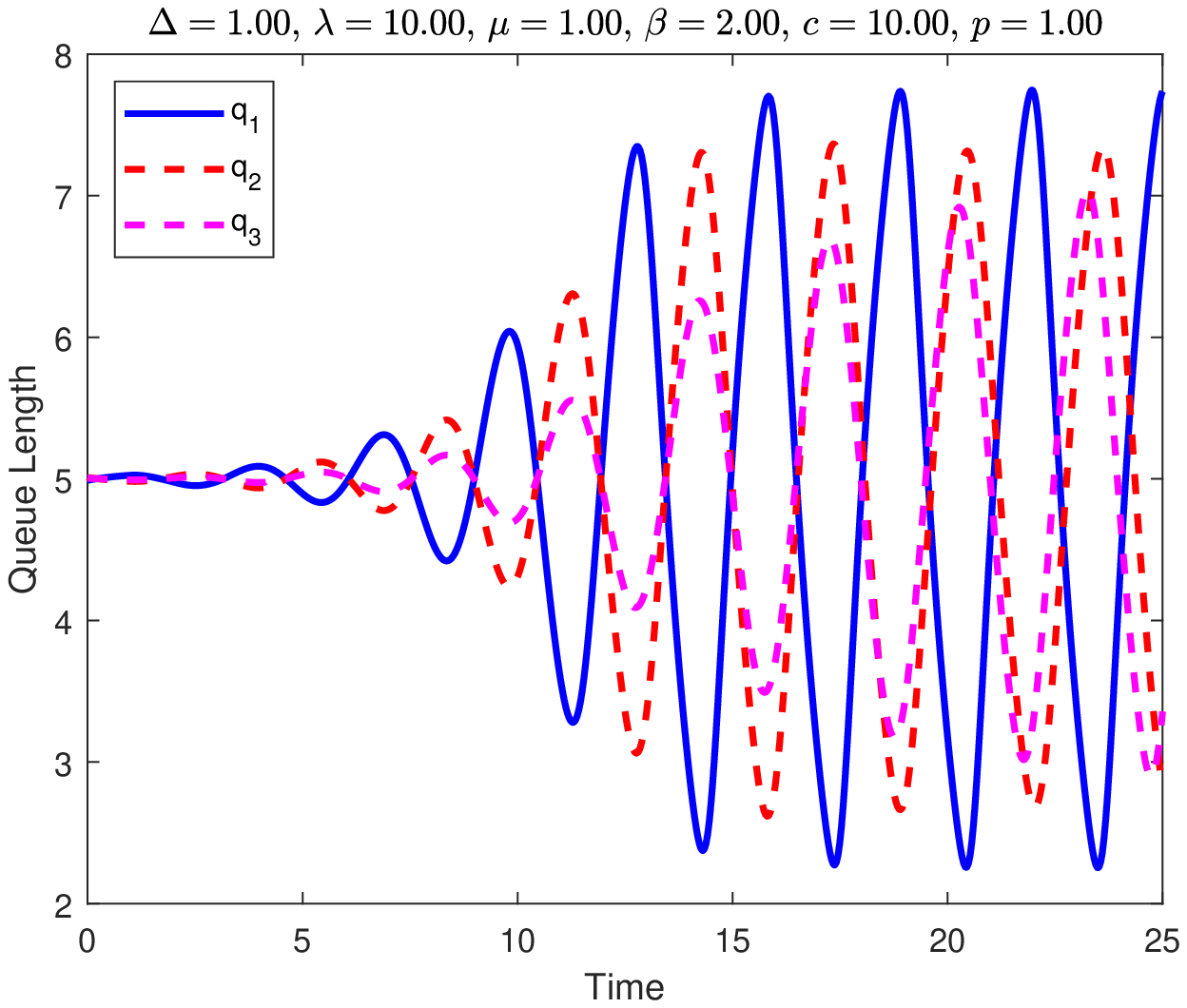}\includegraphics[scale=.33]{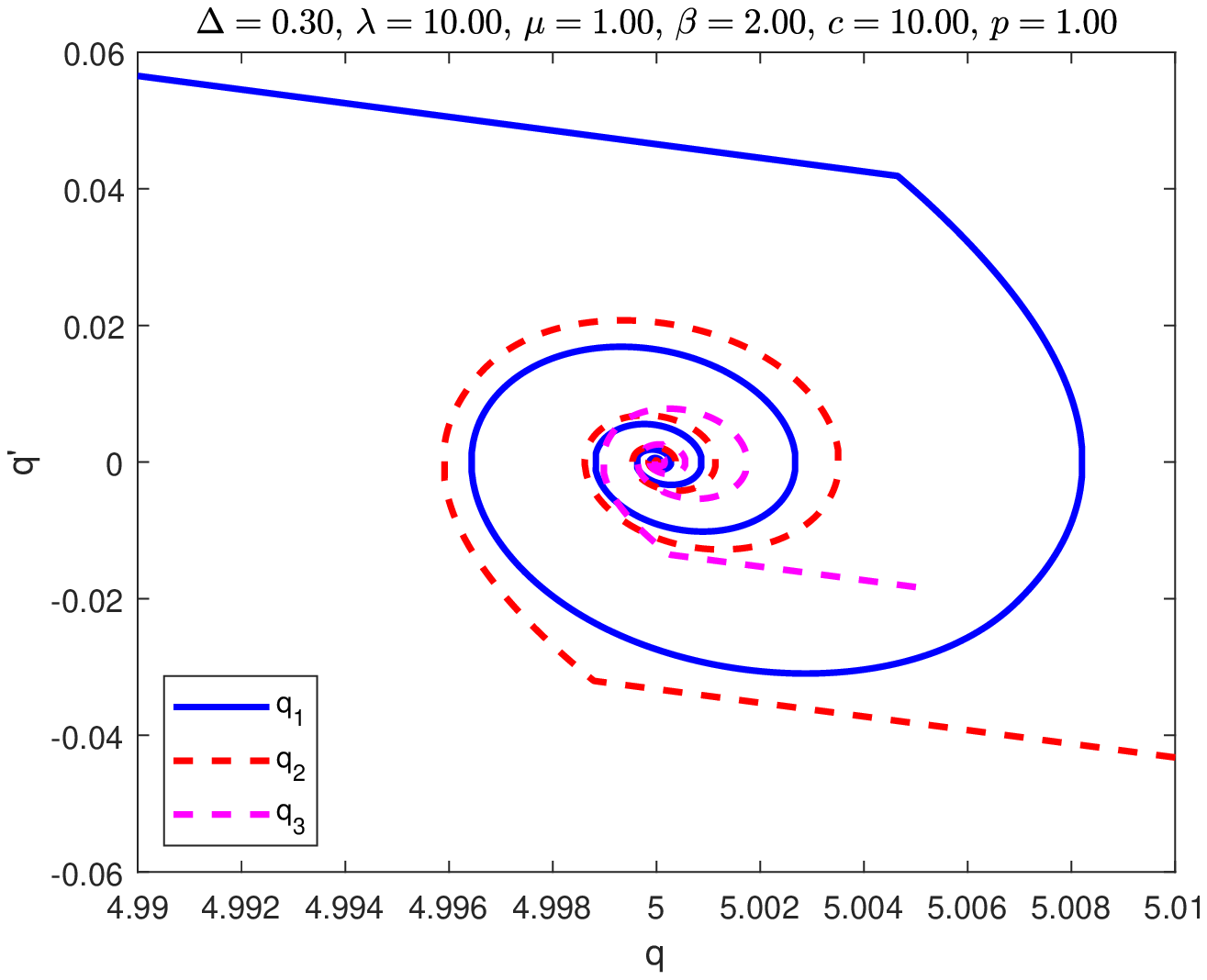}\includegraphics[scale=.33]{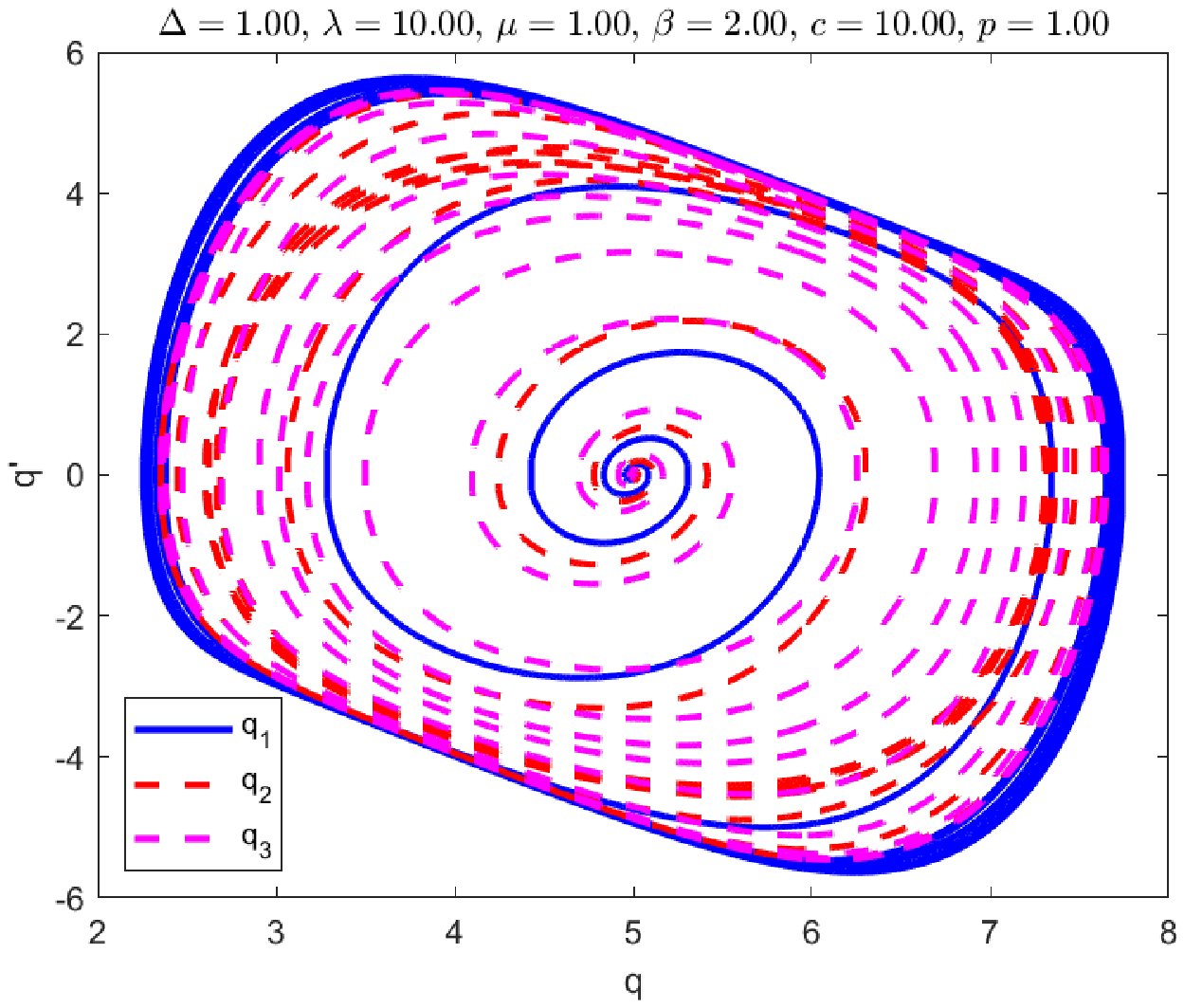}
\caption{Before and after the change in stability in the case where $\lambda f(0) < \mu c$ with constant history function on $[-\Delta, 0]$ with $q_1 = 4.99$, $q_2 = 5.01$, and $q_3 = 5.005$, $N = 3, \lambda = 10$, $\mu = 1$, $\beta=2$, $c=10$, $p = 1$, $f(x) = \int^{\infty}_{x} \frac{1}{\sqrt{2\pi}} e^{-y^2/2}dy$. The left two plots are queue length versus time with $\Delta = .3$ (Left) and $\Delta = 1$ (Right). The right two plots are phase plots of the queue length derivative with respect to time against queue length for $\Delta = .3$ (Left) and $\Delta = 1$ (Right).}
\label{fig9}
\end{figure}

\begin{figure}[ht!]
\hspace{-15mm}\includegraphics[scale=.33]{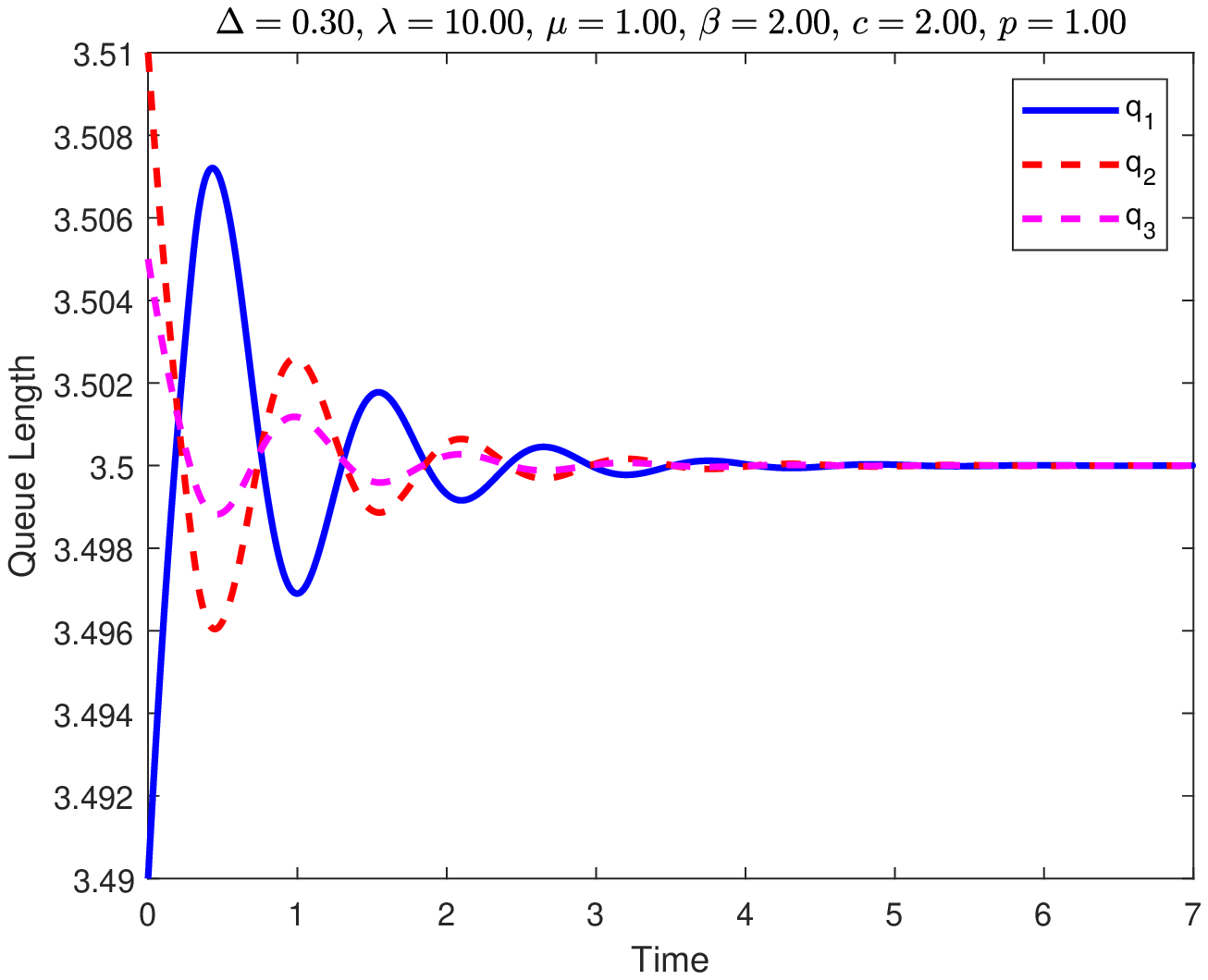}\includegraphics[scale=.33]{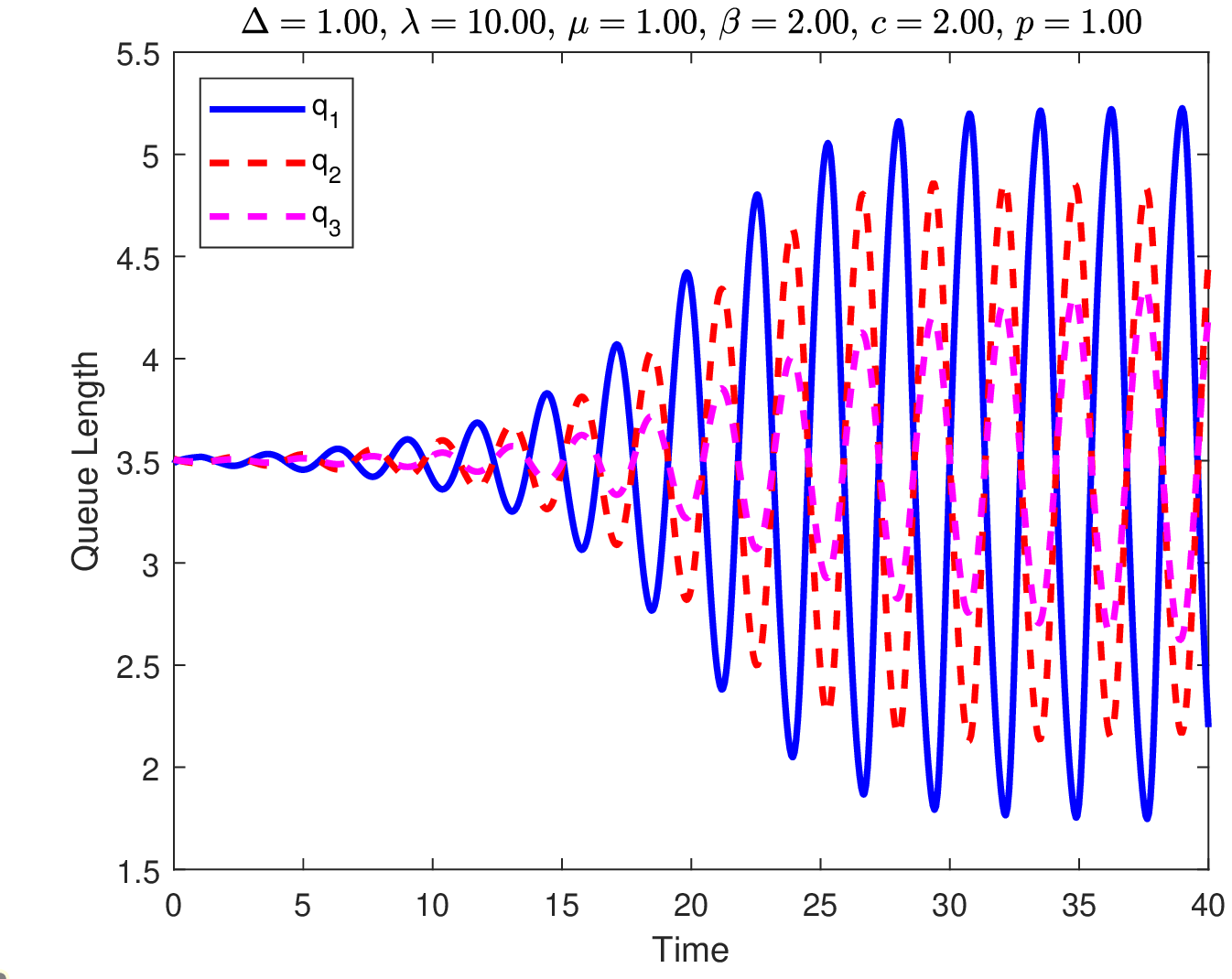}\includegraphics[scale=.33]{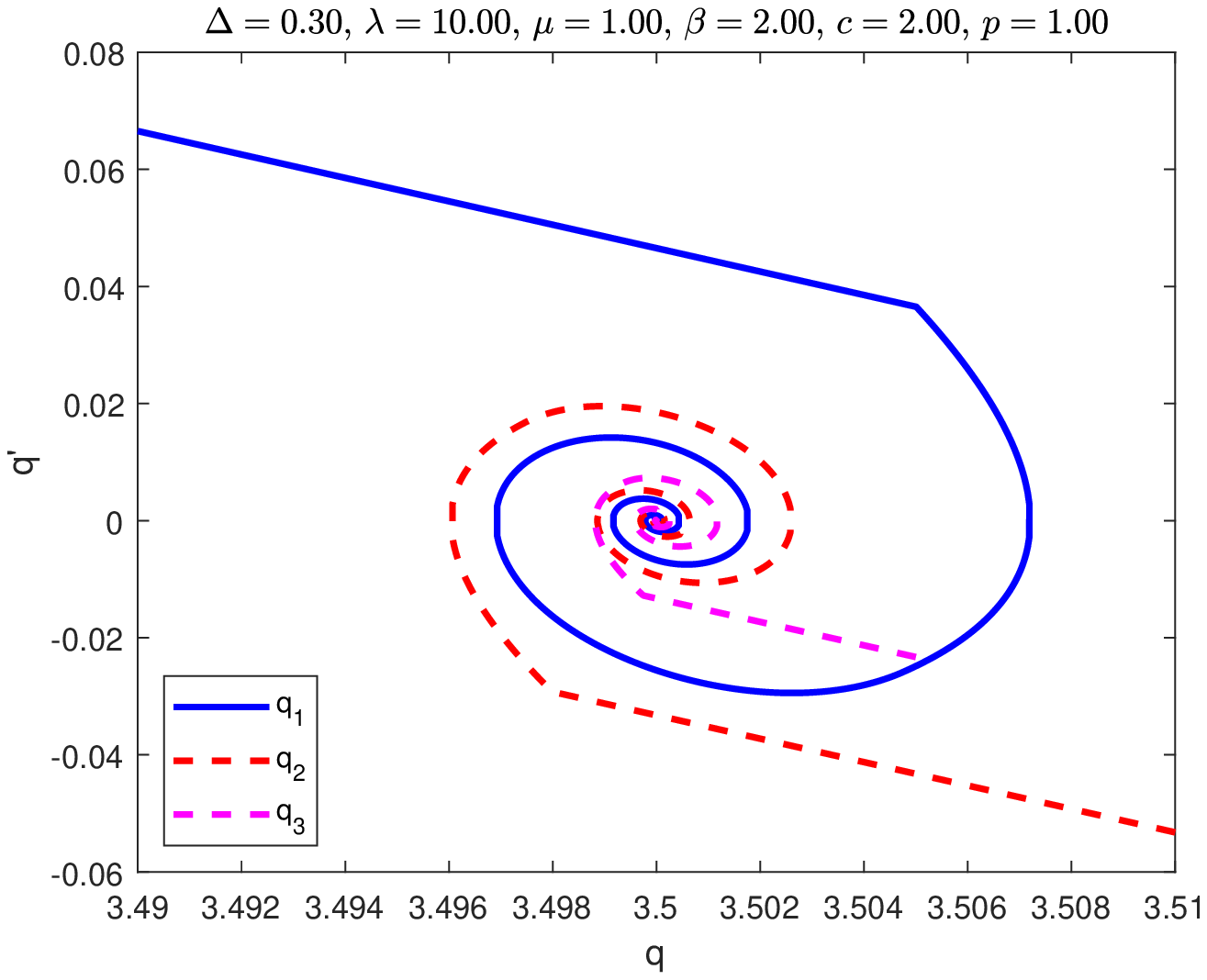}\includegraphics[scale=.33]{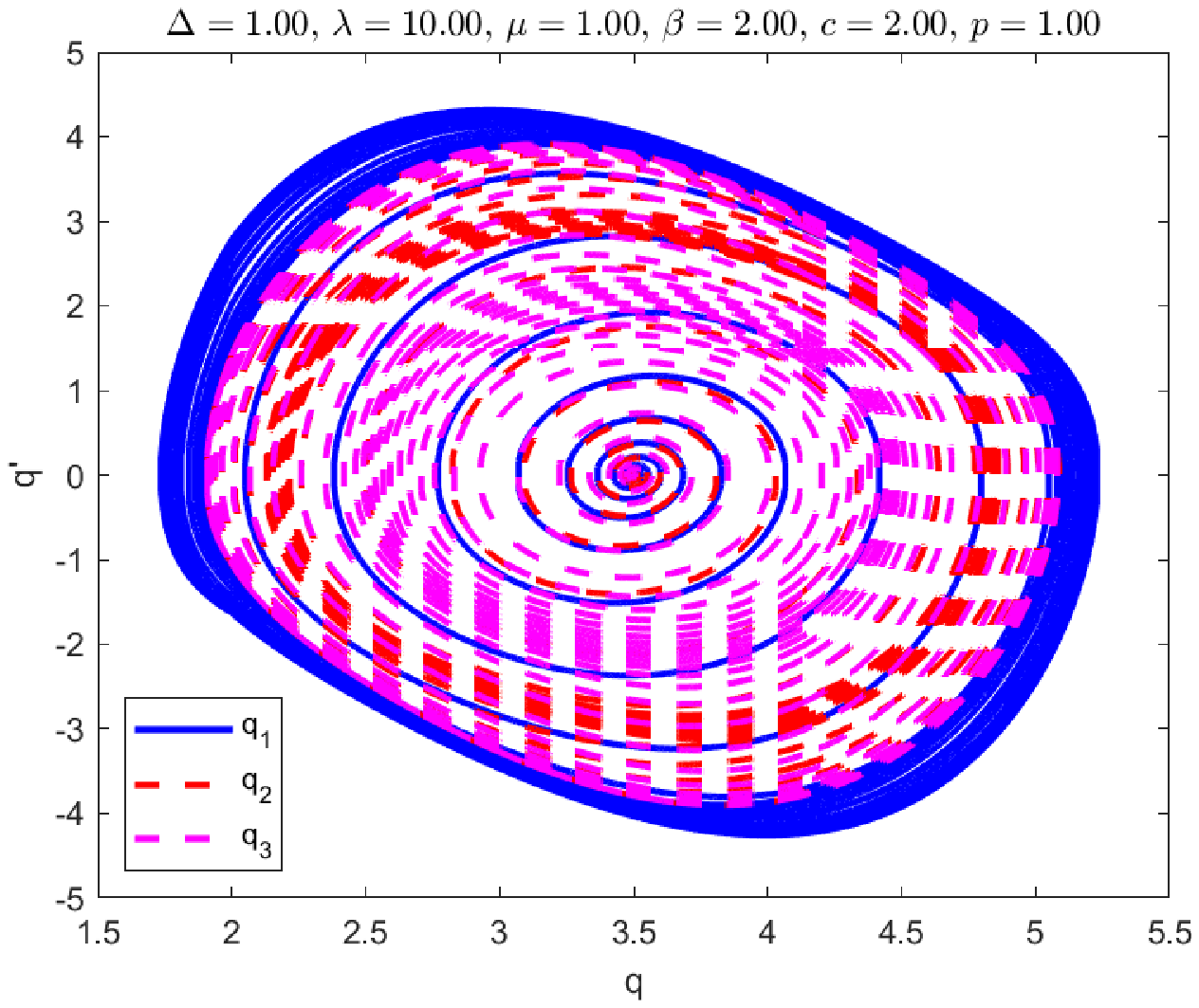}
\caption{Before and after the change in stability in the case where $\lambda f(0) > \mu c$ with constant history function on $[-\Delta, 0]$ with $q_1 = 4.99$, $q_2 = 5.01$, and $q_3 = 5.005$, $N = 3, \lambda = 10$, $\mu = 1$, $\beta=2$, $c=2$, $p = 1$, $f(x) = \int^{\infty}_{x} \frac{1}{\sqrt{2\pi}} e^{-y^2/2}dy$. The left two plots are queue length versus time with $\Delta = .3$ (Left) and $\Delta = 1$ (Right). The right two plots are phase plots of the queue length derivative with respect to time against queue length for $\Delta = .3$ (Left) and $\Delta = 1$ (Right).}
\label{fig10}
\end{figure}

\begin{figure}[ht!]
\hspace{-15mm}\includegraphics[scale=.45]{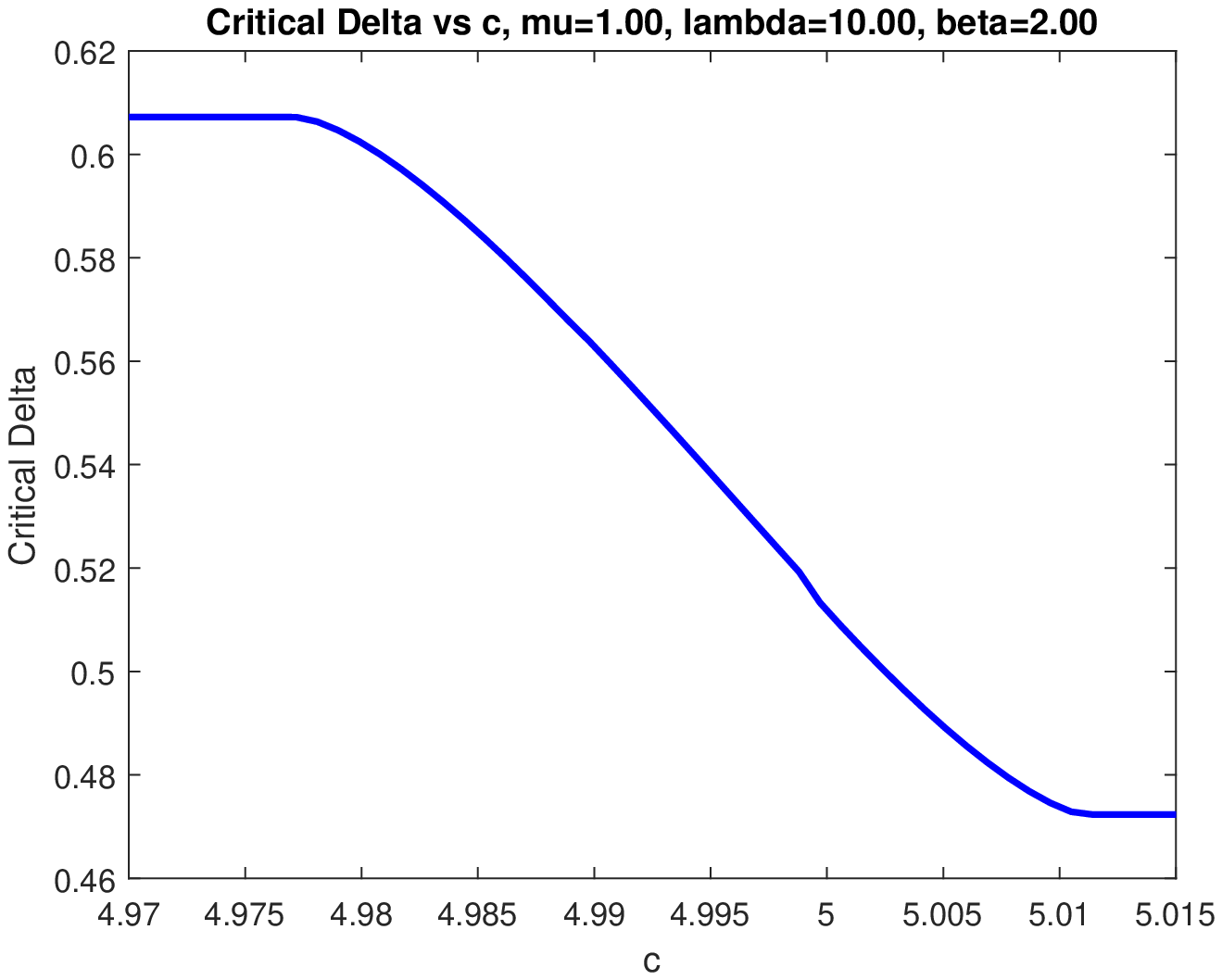}\includegraphics[scale=.45]{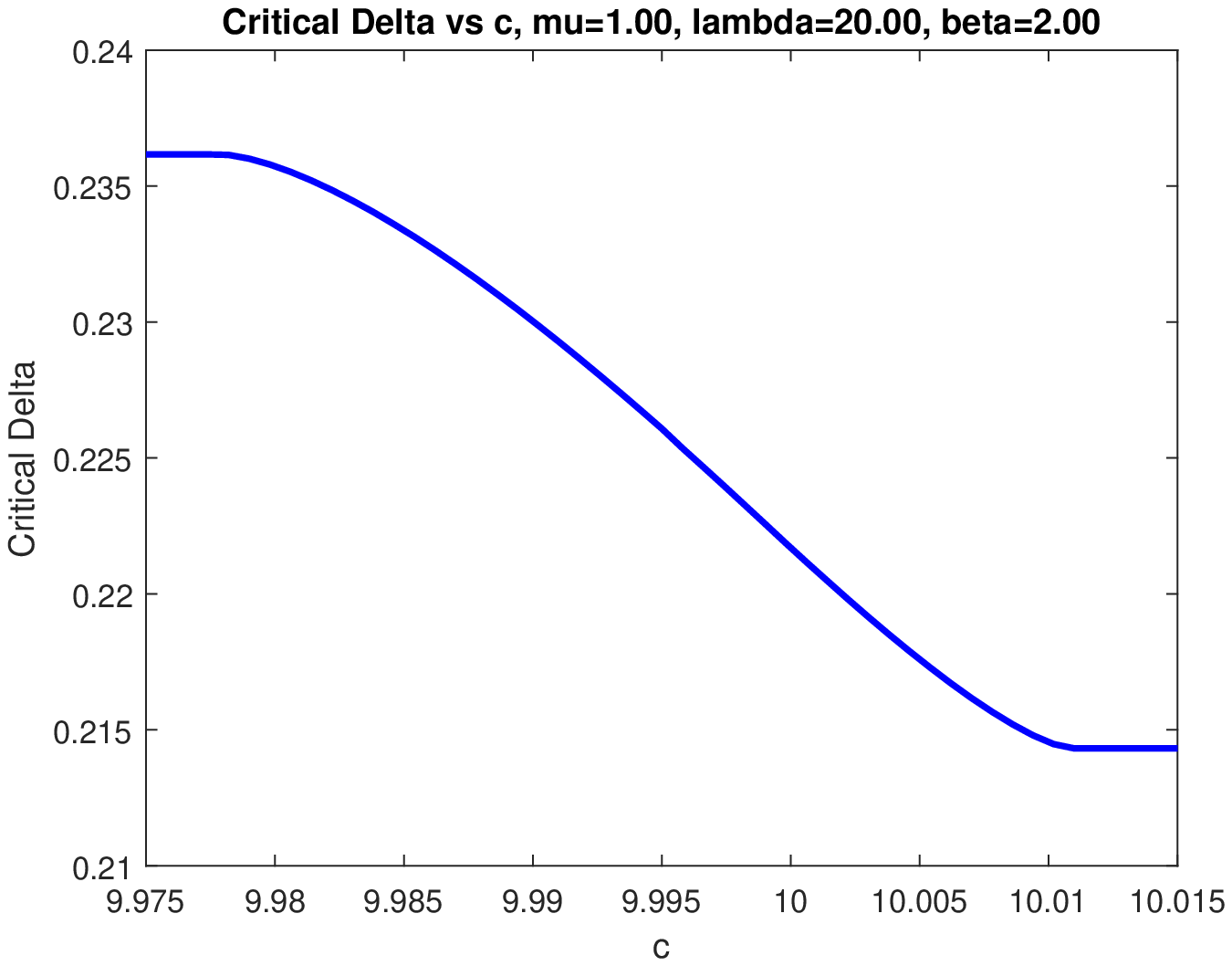}\includegraphics[scale=.45]{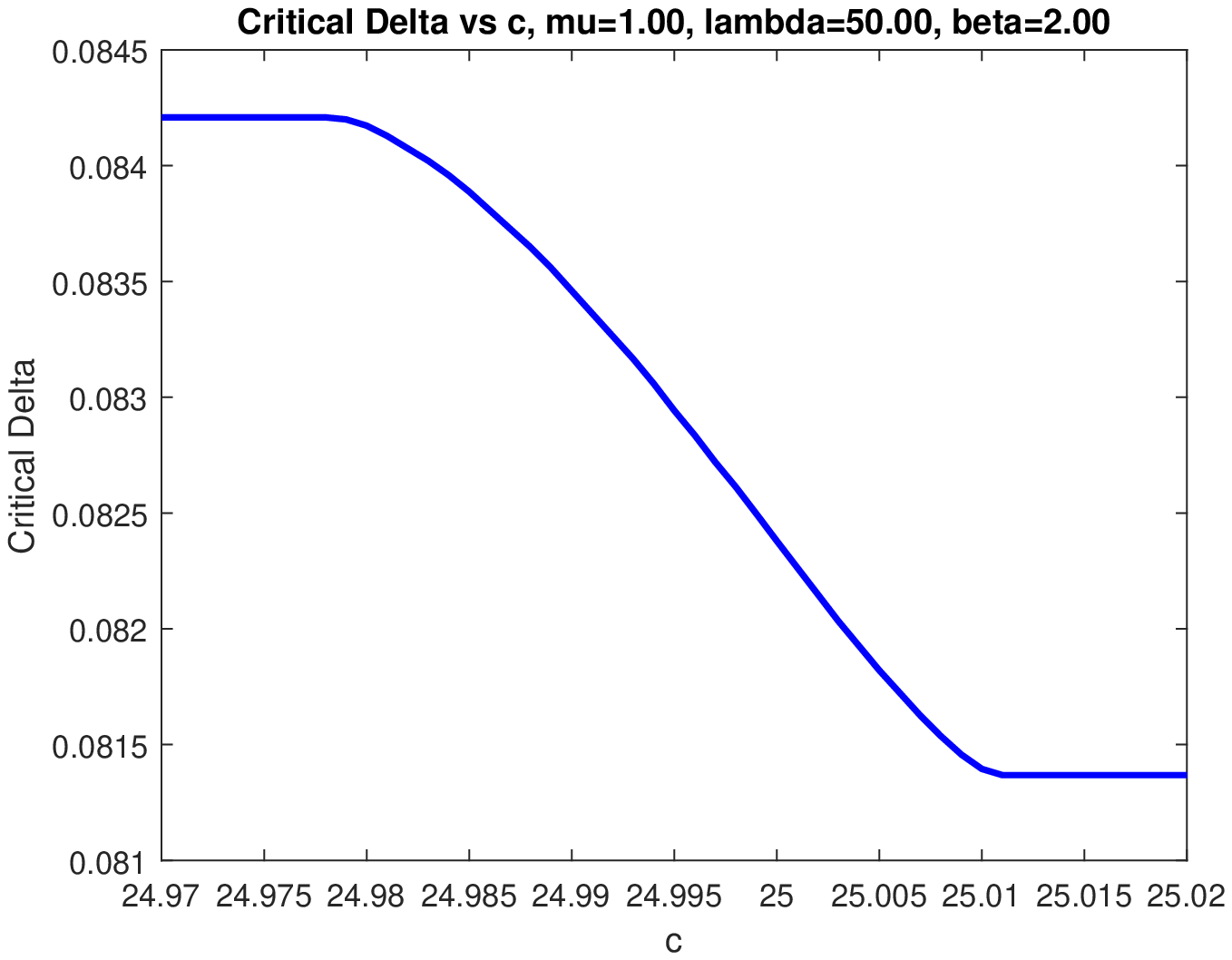}
\caption{How $\Delta_{\text{cr}}$ varies with $c$ and $f(x) = \int^{\infty}_{x} \frac{1}{\sqrt{2\pi}} e^{-y^2/2}dy$ keeping the other parameters fixed with $\mu = 1, \beta = 2,$ and $\theta = 2$ for $\lambda = 10$ (Left), $\lambda = 20$ (Middle), and $\lambda = 50$ (Right). The parameter region is crossed when $c = \frac{\lambda f(0)}{\mu}$. The values of $\Delta_{\text{cr}}$ corresponding to the leftmost and rightmost values of $c$ plotted are $\Delta_{\text{cr}} = 0.6072$ and $\Delta_{\text{cr}} = 0.4723$ (Left), $\Delta_{\text{cr}} = 0.2362$ and $\Delta_{\text{cr}} = 0.2143$ (Middle), and $\Delta_{\text{cr}} = 0.0842$ and $\Delta_{\text{cr}} = 0.0814$ (Right).}
\label{fig11}
\end{figure}

\begin{figure}[ht!]
\hspace{-15mm}\includegraphics[scale=.45]{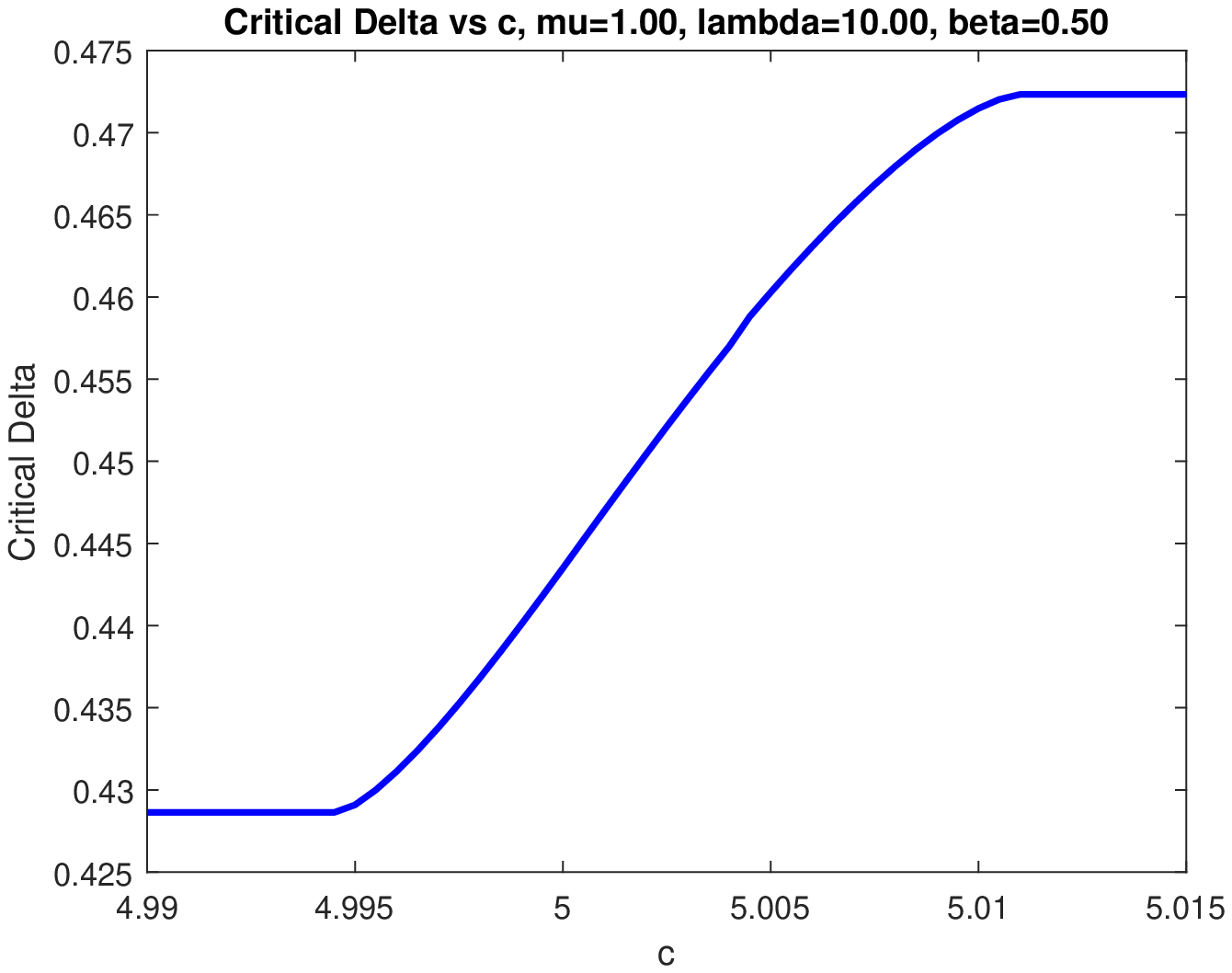}\includegraphics[scale=.45]{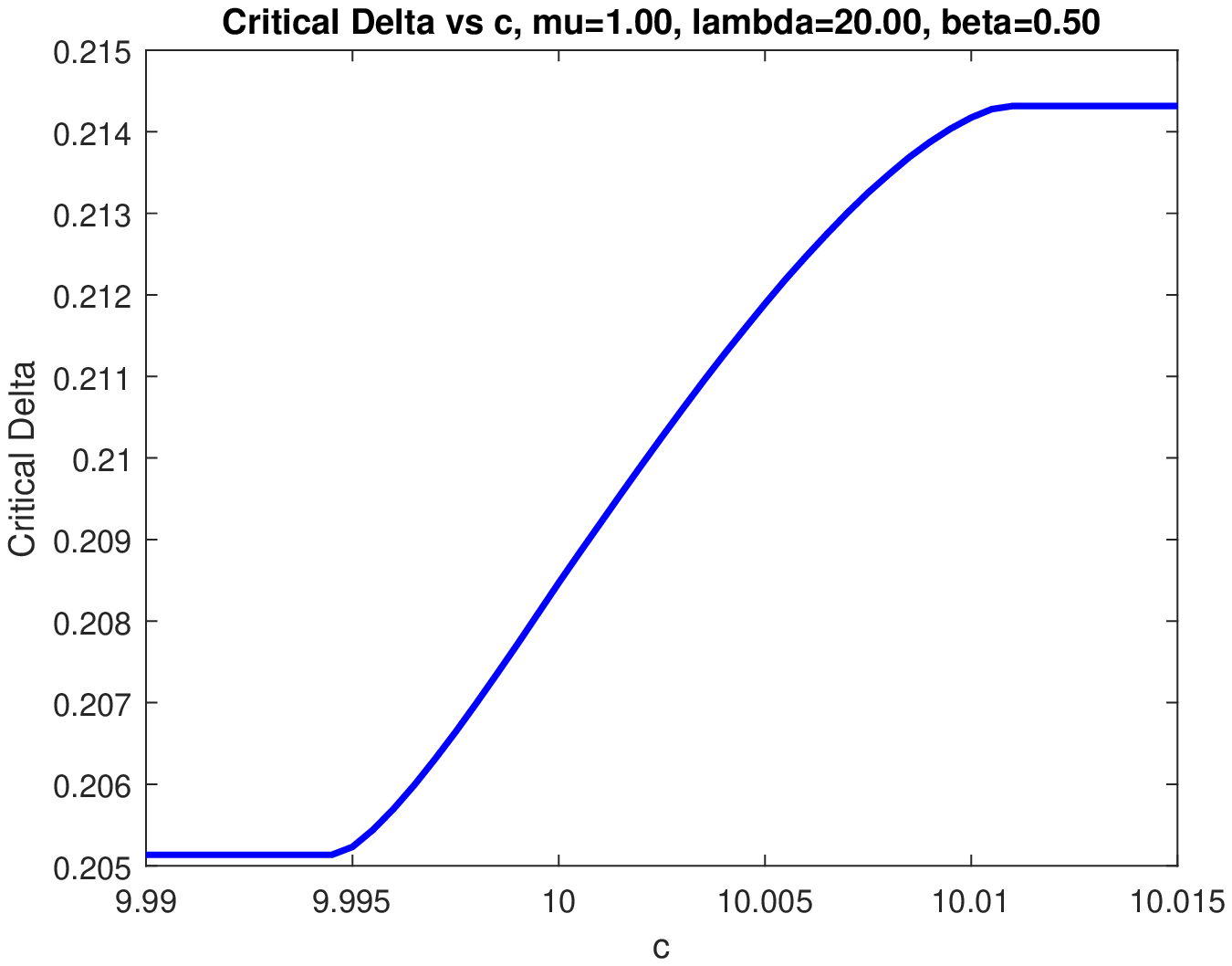}\includegraphics[scale=.45]{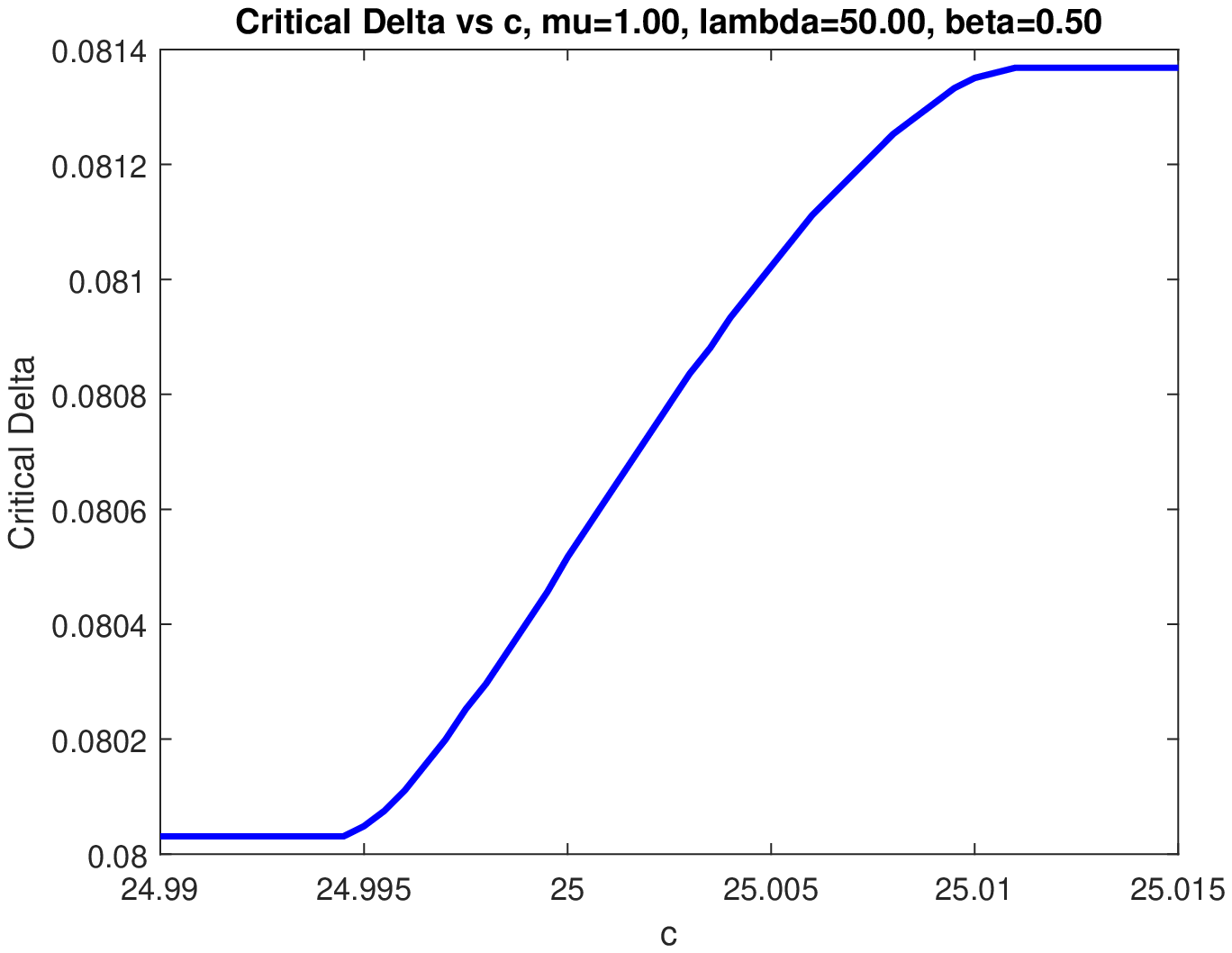}
\caption{How $\Delta_{\text{cr}}$ varies with $c$ and $f(x) = \int^{\infty}_{x} \frac{1}{\sqrt{2\pi}} e^{-y^2/2}dy$ keeping the other parameters fixed with $\mu = 1, \beta = \frac{1}{2},$ and $\theta = 2$ for $\lambda = 10$ (Left), $\lambda = 20$ (Middle), and $\lambda = 50$ (Right). The parameter region is crossed when $c = \frac{\lambda f(0)}{\mu}$. The values of $\Delta_{\text{cr}}$ corresponding to the leftmost and rightmost values of $c$ plotted are $\Delta_{\text{cr}} = 0.4286$ and $\Delta_{\text{cr}} = 0.4723$ (Left), $\Delta_{\text{cr}} = 0.2051$ and $\Delta_{\text{cr}} = 0.2143$ (Middle), and $\Delta_{\text{cr}} = 0.0800$ and $\Delta_{\text{cr}} = 0.0814$ (Right).}
\label{fig12}
\end{figure}



\section{Conclusion} \label{conclusion_section}

In this paper, we analyze the stability of a new queueing model where the queues are balanced via a mean field interaction with a time delay. From our model, we obtain a system of $N$ delay differential equations for a system of $N$ queues. Additionally, our queueing model allows for customer abandonment which introduces a point of non-differentiability that splits the parameter space into two regions. We derive an exact expression for the critical delay of this system in each of the two parameter regions where performing a linearization analysis is valid and we consider numerically what happens to the value of the critical delay as the parameters are varied to cross between the two regions. 

There are several problems that remain for future research. For example, it would be interesting to compute the amplitude of the oscillations that are present when the delay is larger than the critical delay.  The calculation of the amplitude would be useful for quantifying how detrimental the oscillations are to the efficiency of the queueing network. Approximating the amplitude of the oscillations near the critical delay via perturbation methods is a possible approach \citet{novitzky2019nonlinear}. Additionally, it may be interesting to consider other types of information such as the update model of \citet{novitzky2020update}.

\section{Acknowledgements}

We would like to thank the Center for Applied Mathematics at Cornell University for sponsoring Philip Doldo’s research. Finally, we acknowledge the gracious support of the National Science Foundation (NSF) for Jamol Penders Career Award CMMI \# 1751975.

\bibliographystyle{plainnat}
\bibliography{main}

\end{document}